\documentclass[12pt]{amsart}
\usepackage{amssymb,amscd}
\usepackage{verbatim}

\usepackage{xcolor}

\usepackage{amsmath,amssymb,graphicx,mathrsfs}   
\usepackage{enumerate}
\usepackage[colorlinks=true,allcolors = blue]{hyperref} 

\usepackage{tikz}
\usetikzlibrary{matrix}

\usepackage[all]{xy}

\usepackage[title]{appendix}

\usepackage{amssymb,amsfonts,amsthm,amsmath,calligra}
\usepackage{slashed}
\usepackage{yfonts}
\usepackage{mathrsfs,pifont}
\usepackage{float}

\usepackage[all]{xy}

\usepackage{tikz}

\usepackage{graphicx}
\usepackage{xcolor}

\usepackage{amssymb,amsfonts,amsthm,amsmath}
\usepackage[all]{xy}
\usepackage{slashed}
\usepackage{yfonts}
\usepackage{mathrsfs,pifont}

\let\frak\mathfrak

\def\>{\relax\ifmmode\mskip.666667\thinmuskip\relax\else\kern.111111em\fi}
\def\<{\relax\ifmmode\mskip-.333333\thinmuskip\relax\else\kern-.0555556em\fi}
\def\vsk#1>{\vskip#1\baselineskip}
\def\vv#1>{\vadjust{\vsk#1>}\ignorespaces}
\def\vvn#1>{\vadjust{\nobreak\vsk#1>\nobreak}\ignorespaces}

  \let\ssize\scriptstyle
\let\sssize\scriptscriptstyle

\let\Medskip\medskip
\def\medskip{\par\Medskip}
\let\Bigskip\bigskip
\def\bigskip{\par\Bigskip}

\let\Maketitle\maketitle
\def\maketitle{\Maketitle\thispagestyle{empty}\let\maketitle\empty}

\newtheorem{thm}{Theorem}[section]
\newtheorem{cor}[thm]{Corollary}
\newtheorem{lem}[thm]{Lemma}
\newtheorem{prop}[thm]{Proposition}

\newtheorem{defn}[thm]{Definition}

\theoremstyle{definition}                                  
\numberwithin{equation}{section}

\theoremstyle{definition}
\newtheorem*{rem}{Remark}
\newtheorem*{example}{Example}

\let\mc\mathcal
\let\nc\newcommand

\let\la\lambda

\let\phi\varphi

\let\der\partial

\let\geq\geqslant

\let\leq\leqslant

\let\on\operatorname
\let\bi\bibitem
\let\bs\boldsymbol

\def\C{{\mathbb C}}
\def\Z{{\mathbb Z}}

\def\F{{\mathbb F}}

\def\+#1{^{\{#1\}}}

\def\diag{\on{diag}}

\def\beq{\begin{equation}}
\def\eeq{\end{equation}}
\def\be{\begin{equation*}}
\def\ee{\end{equation*}}

\nc{\bea}{\begin{eqnarray*}}
\nc{\eea}{\end{eqnarray*}}
\nc{\bean}{\begin{eqnarray}}
\nc{\eean}{\end{eqnarray}}

\nc{\Il}{{\mc I_{\bs\la}}}
\nc{\bla}{{\bs\la}}
\nc{\Fla}{\F_\bla}
\nc{\tfl}{{T^*\Fla}}
\nc{\GL}{{GL_n(\C)}}
\nc{\GLC}{{GL_n(\C)\times\C^*}}

\let\sd s 

\def\ddk_#1{\kk_{#1}\<\>\frac\der{\der\<\>\kk_{#1}}}

\def\bul{\mathbin{\raise.2ex\hbox{$\sssize\bullet$}}}
\def\intt{\mathchoice
{\mathop{\raise.2ex\rlap{$\,\,\ssize\backslash$}{\intop}}\nolimits}
{\mathop{\raise.3ex\rlap{$\,\sssize\backslash$}{\intop}}\nolimits}
{\mathop{\raise.1ex\rlap{$\sssize\>\backslash$}{\intop}}\nolimits}
{\mathop{\rlap{$\sssize\<\>\backslash$}{\intop}}\nolimits}}

\let\kk q 
\let\cc c

\let\Ko K

\def\GZ/{Gelfand-Zetlin}
\def\KZ/{{\slshape KZ\/}}
\def\qKZ/{{\slshape qKZ\/}}
\def\XXX/{{\slshape XXX\/}}

\def\Fr{\mathsf{U}}

\def\aa{\bs{a}}
\def\uu{\bs{u}}
\def\zz{{\bs z}}

\nc{\A}{{\mc A}}

\def\uu{{\bs u}}

\nc{\hsl}{\widehat{{\frak{sl}_2}}}

\nc{\BC}{{ \mathbb C}}
\nc{\lra}{\longrightarrow}
\nc{\CO}{{\mathcal{O}}}
\nc{\BZ}{{ \mathbb Z}}
\nc{\hfn}{\hat{\frak{n}}}
\nc\Zs{{\Z/p^s\Z}}
\nc\Zo{{\Zs[z]^0}}
\nc\gr{{\on{gr}}}

\nc\fD{{\frak D}}

\newcommand{\matC}{\mathbb{C}}
\newcommand{\matN}{\mathbb{N}}

\newcommand{\matQ}{\mathbb{Q}}
\newcommand{\matP}{\mathbb{P}}
\newcommand{\matZ}{\mathbb{Z}}

\newcommand{\Mop}{\mathbf{M}}

\newcommand{\cL}{\mathcal{L}}

\usepackage{tikz}

\usetikzlibrary{decorations}
\usetikzlibrary{decorations.pathmorphing}
\usetikzlibrary{calc}

\newsavebox{\hweights}
\savebox{\hweights}{%
\begin{tikzpicture}[baseline= {($(current bounding box.base)-(0pt,-30pt)$)}]
\begin{scope}
\draw[line width=0.35mm] (0,0)-- (5,5) ;

\draw[line width=0.35mm] (-1,1)-- (4,6) ;

\draw[line width=0.35mm] (-2,2)-- (3,7) ;

\draw[line width=0.35mm] (-3,3)-- (2,8) ;

\draw[line width=0.35mm] (0,0)-- (-3,3) ;
\draw[line width=0.35mm] (1,1)-- (-2,4) ;
\draw[line width=0.35mm] (2,2)-- (-1,5) ;
\draw[line width=0.35mm] (3,3)-- (0,6) ;
\draw[line width=0.35mm] (4,4)-- (1,7) ;
\draw[line width=0.35mm] (5,5)-- (2,8) ;
\node at (0,1) { $1$};
\node at (-1,2) { $2$}; \node at (1,2) { $2$};
\node at (-2,3) { $3$};\node at (0,3) { $3$};
\node at (2,3) { $3$};
\node at (-1,4) { $4$};
\node at (1,4) { $4$};
\node at (3,4) { $4$};
\node at (0,5) { $5$};
\node at (2,5) { $5$};
\node at (4,5) { $5$};
\node at (1,6) { $6$};
\node at (3,6) { $6$};
\node at (2,7) { $7$};
\end{scope}
\end{tikzpicture}}

\begin{document}

\hrule width0pt
\vsk->

\title[Frobenius intertwiners for $q$-difference equations]
{Frobenius intertwiners for $q$-difference equations}

\author
[Andrey Smirnov ]
{Andrey Smirnov$^{\star}$ }

\maketitle

\begin{center}
{ Department of Mathematics, University
of North Carolina at Chapel Hill\\ Chapel Hill, NC 27599-3250, USA\/}

\end{center}

\vsk>
{\leftskip3pc \rightskip\leftskip \parindent0pt \Small
{\it Key words\/}:  Frobenius structures, $q$-hypergeometric functions, $p$-adic $q$-Gamma function, Nakajima varieties, enumerative geometry

\vsk.6>
{\it 2020 Mathematics Subject Classification\/}: 14G33 (11D79, 32G34, 33C05, 33E30)
\par}


{\let\thefootnote\relax
\footnotetext{\vsk-.8>\noindent
$^\star\<${\sl E\>-mail}:\enspace asmirnov@email.unc.edu
}

\begin{abstract}
We consider a class of $q$-hypergeometric equations describing the quantum difference equation for  the cotangent bundles over projective spaces $X=T^{*}\matP^{n-1}$ . We show that over $\matQ_p$ these equations are equipped with the Frobenius action $(q,z)\to (q^p,z^p)$. We obtain an explicit formula for the constant term of the Frobenius intertwiner in terms of   the $p$-adic $q$-gamma function of Koblitz. In the limit $q\to 1$ we arrive at the Frobenius structures for the $p$-adic hypergeometric and Bessel differential equations  studied by Dwork. In particular, we find closed formulas for $p$-adic constants appearing in works of Dwork and Sperber in terms of $p$-adic zeta functions.
 \end{abstract}


\setcounter{footnote}{0}

\section{Introduction} 
\subsection{$q$-difference equations and enumerative geometry} 
The quantum K-theory of quiver varieties \cite{Oko17} is governed by the quantum difference equations (QDE) \cite{OS22}. It is expected that over $\mathbb{Q}_p$ these equations are equipped with a symmetry which acts on the arguments of the equations by $z\to z^p$. In the limit $q\to 1$ these equations reduce to the {\it quantum differential equations} of quiver varieties \cite{MO19} while the symmetry $z\to z^p$ specializes to the Frobenius structures, which are well known in the theory of $p$-adic differential equations \cite{Dw89,Ked21}.  In this sense, the expected symmetry of QDE may be viewed as a $q$-deformation of the Frobenius structures. 

One may hope that the $q$-deformed Frobenius structures arising this way may be of significance in arithmetic geometry, see \cite{Sch,BS}. 

\subsection{$q$-hypergeometric case}
In this paper consider the  $q$-deformed Frobenius structure for simplest examples of quiver varieties given by the cotangent bundles over projective spaces $X=T^{*} \mathbb{P}^{n-1}$.  The cotangent bundle $X$ is equipped with a natural action of a torus $T=(\mathbb{C}^{*})^{n+1}$. We denote by $h$ and $\aa=(a_1,\dots, a_n)$ the corresponding equivariant parameters of $T$. 
The parameter $h$ denotes the character of the one-dimensional representation $\mathbb{C} \omega$ spanned by the symplectic form $\omega \in H^{2}(X,\matC)$. This parameter plays a special role in the present paper.

In Section \ref{qhypersec} we describe the QDE of $X$ which is nothing but the very classical $q$-hypergeometric equation with parameters $h$ and $\aa$.

\subsection{$q$-deformed Frobenius intertwiner}
In the context of enumerative geometry it is standard to represent the fundamental solution matrix of QDE as a power series with coefficients in the equivariant $K$-theory:
\bean \label{fudmat}
\Psi(h,\aa,q,z)  \in  K_{T}(X)^{\otimes 2}[[z]]
\eean
In Section \ref{qhypersec} we fix such $\tilde{\Psi}(h,\aa,q,z)$ where $\tilde{}$  refers to a certain specific choice of the normalization, which involves $q$-gamma functions. The key property of this normalization, relevant to our work, is that $ \tilde{\Psi}(h, \aa, q, z)$ satisfies a set of regular difference equations in the equivariant parameters $h$ and $\aa$; see Theorem \ref{shift_thm}.

We consider the power series
\bean \label{frobdefin}
\Fr(h,\aa,q,z) = \tilde{\Psi}(p h,p\aa,q,z) \tilde{\Psi}( h,\aa,q^{p},z^p)^{-1} \in  K_{T}(X)^{\otimes 2}[[z]]
\eean
Fixing a basis in $K_{T}(X)$ we can represent $\Fr(h,\aa,q,z)$ by an $n\times n$-matrix with coefficients given by power series in $z$.  
Assume that 
$$
|q-1|_p<1, \ \  \textrm{and} \ \  (h,a_1,\dots,a_n) \in \mathbb{Z}^{n+1}_p
$$ 
We show that under these assumptions (Theorem \ref{mainthm}), $\Fr(h,\aa,q,z) \in \mathsf{GL}_n(\hat{E}_p)$ where $\hat{E}_p$ denotes the field of $p$-adic analytic functions. This means that $\Fr(h,\aa,q,z)$ is a Frobenius intertwiner between the $q$-hypergeometric systems with parameters $(h,\aa,q^p,z^p)$ and $(ph,p\aa,q,z)$.

\vspace{2mm}
\noindent
We observe that the Frobenius intertwiner $\Fr(h,\aa,q,z)$ has simple specializations at $q$ given by $p$-adic unit roots. If $q$ is a $p$-adic unit root of order $p^s$ then 
$$
|q-1|_p =p^{-\frac{1}{p^{s-1}(p-1)}}<1
$$
We show that if $q$ is such a root (Theorem \ref{newthm}) then the Frobenius intertwiner $\Fr(h,\aa,q,z)$ specializes to a rational function of variables $z$ and
$
q^{p h}, \ \  q^{p a_1}, \dots, q^{p a_n}.
$
In particular, these rational functions satisfy
$$
\Fr(h,\aa,q,z)= \Fr(h^{(s)},\aa^{(s)},q,z)
$$
where $\aa^{(s)}=(a_1^{(s)},\dots, a^{(s)}_n)$ and $h^{(s)}$ are natural numbers defined  by $h^{(s)} = h \pmod {p^{s-1}}$, $a^{(s)}_i = a_i \pmod {p^{s-1}}$, $i=1,\dots, n$.
This property can be considered as a $q$-version of the Dwork's congruences  for the classical hypergeometric functions \cite{Dw69} and the Hasse-Witt matrices \cite{VZ}.

\subsection{The constant term of the Frobenius intertwiner} 
In Section \ref{frobintsect}, we give an explicit formula for the constant term of the Frobenius intertwiner (\ref{frobdefin}), i.e., its specialization at $z=0$.  First, for a vector bundle $\mathcal{V}$ over $X$ we define a $K$-theory valued characteristic class
$\Gamma_{p,q}(\mathcal{V})\in K_{T}(X)$. 
This class is constructed from the $p$-adic $q$-gamma function $\Gamma_{p,q}$ defined by Koblitz \cite{Ko80}. Second, we show that $\Fr(h,\aa,q,0)$ is an operator acting in $K_{T}(X)$ as the operator of multiplication by the $K$-theory class
\bean \label{contermform}
\Fr(h,\aa,q,0)=\Gamma_{p,q}((q^p-q^{p h})T^{1/2}X^{(p)})
\eean
where $T^{1/2}X$ denotes the polarization bundle of $X$ and $T^{1/2}X^{(p)}= \psi^p(T^{1/2}X)$ its twist by the $p$-th Adams operation. 

For generic values of the equivariant parameters $a_i\neq a_j$, the equivariant $K$-theory $K_{T}(X)$ has a distinguished basis given by the $K$-theory classes of $T$-fixed points. In this basis the operator 
$\Fr(h,\aa,q,0)$ is diagonal, with the eigenvalues 
$$
\Fr(h,\aa,q,0)_{i,i}=  \prod\limits_{j=1}^{n} \, \dfrac{\Gamma_{p,q}(p a_i+ p h - p a_j )}{\Gamma_{p,q}(p a_i-p a_j ) \Gamma_{p,q}(p h )}, \ \  i=1,\dots, n.
$$
We expect that (\ref{contermform}) holds for the constant term of the Frobenius intertwiners for more general varieties. 

\subsection{Limits and specializations}
The $q$-deformed Frobenius structure has several important specializations. First, the limit $q\to 1$ corresponds to reduction of the equivariant $K$-theory to the equivariant cohomology. In this limit the QDE specializes to the  hypergeometric differential equation
- the quantum differential equation of $X = T^{*} \mathbb{P}^{n-1}$. The $p$-adic $q$-gamma function $\Gamma_{p,q}$ specializes to the $p$-adic gamma function $\Gamma_p$ of Morita~\cite{Mo75}.  Combining this together we find that $\left.\Fr(h,\aa,q,z)\right|_{q=1}$ is the Frobenius structure for the hypergeometric differential equation. Its constant term is the operator acting in the equivariant cohomology $H^{*}_T(X)$ as  multiplication by the cohomology class
$\Gamma_p(TX^{(p)})$ where $TX^{(p)}$ denotes the tangent bundle of $X$ twisted by $p$-th Adams operation.
For generic values of the equivariant parameters this operator is diagonal in the basis of $T$-fixed points with eigenvalues 
\bean \label{conthyperU}  \ \ \ \ \ \ 
\begin{small}
\Fr(h,\aa,1,0)_{i,i}= \prod\limits_{w \in \textrm{char} (T_i X)}\, \Gamma_{p}(p w) =  \prod\limits_{j=1}^{n} \, \Gamma_p(p a_i +p h -p a_j ) \Gamma_p(p a_j -  p a_i )\ \ \
i=1,\dots, n.
\end{small}
\eean
where the first product is over the $T$-characters appearing in the tangent space to $X$ at $i$-th $T$-fixed point. In this way we reproduce the results of Dwork \cite{Dw69,Dw89} where the Frobenius intertwiners for the hypergeometric equations were discovered and also give them a geometric interpretation.  We also refer to \cite{Ked21} where a formula equivalent to (\ref{conthyperU}) was obtained using different tools.

Second, it is well known that in the limit $h\to \infty$ the quantum differential equation of $X=T^{*} \mathbb{P}^{n-1}$ specializes to quantum differential equation of $\mathbb{P}^{n-1}$. In the most degenerate case $a_1=a_2=\dots = a_n$, this equation coincides with the generalized Bessel equation studied over $\mathbb{Q}_p$ by Sperber in \cite{Sp80}. For $n=2$ this is the very classical Bessel equation investigated by Dwork \cite{Dw74}. As an application, in the limit $h\to \infty$ we obtain an explicit description for the Frobenius structures of Bessel equations. In particular, we obtain closed formulas for certain $p$-adic constants appearing in \cite{Dw74,Sp80} as values of $p$-adic zeta functions.

\subsection{Enumerative definition of Frobenius intertwiner}

In quantum cohomology over fields of positive characteristic one can define new classes of operators known as {\it the quantum Steenrod operations} \cite{F96,W20}.  These operations are defined enumeratively, namely as partition functions counting  stable maps from $\mathbb{P}^1$ with $p$-marked points to $X$.  The stable maps are assumed to be invariant under the cyclic group $\mathbb{Z}/p\mathbb{Z}$, which permutes the marked via rotation of the source $\mathbb{P}^1$.  It was recently shown in \cite{HL24} that for a large family of varieties, the quantum Steenrod operations coincide with the $p$-curvature of the quantum connection of these varieties.

We expect that this story extends naturally to the level of quantum $K$-theory of quiver varieties which is defined via equivariant count of quasimaps \cite{Oko17}. In particular, the $q$-difference equations have natural analogs of $p$-curvature \cite{KS}. In this context,  the quantum Steenrod operations shall be lifted to the $K$-theoretic {\it quantum Adams operations}. Moreover, the Frobenius intertwiner $\Fr(h,\aa,q,z)$ itself can be understood as a {\it partition function counting equivariant quasimaps} from $\mathscr{C}=\mathbb{P}^1$ to $X$ with two relative boundary conditions at $0,\infty \in \mathscr{C}$, where the relative quasimap moduli space at $\infty$ is assumed to be $\mathbb{Z}/p\mathbb{Z}$-equivariant in an appropriate sense.

When $q$ is specialized at $p$-th unit roots, the partition function $\Fr(h,\aa,q,z)$ is an operator which conjugates the $p$-curvature of the $q$-difference connection to the Frobenius twist of this connection. In particular, these objects have the same spectrum \cite{KS}. This hints at the arithmetic significance of $q$-deformed Frobenius intertwiners with $q$ specialized at the unit roots.

We briefly touch on these ideas in Section~\ref{conclsec}.


\section{$q$-hypergeometric functions \label{qhypersec}}
In this section we work over $\mathbb{C}$. In the next section we switch to the field of $p$-adic numbers. 
\subsection{} Let $D$ be the $q$-shift operator acting by $D f(z) = f(zq)$. We consider a $q$-difference equation:
\bean \label{qdedef}
P(\aa,\hbar,z) F(z) =0,
\eean
where $P(\aa,\hbar,z)$ denotes the $q$-difference operator:
\bean \label{qdeoper}
P(\aa,\hbar,z)=\sum_{m=0}^{n} \, \alpha_m (1 - z \hbar^{m})  D^m.
\eean
and $\alpha_m$  denotes the coefficient of $x^m$ in
$\prod\limits_{i=1}^{n}(1-x/u_i)$.
\subsection{}
Let us consider the $q$-hypergeometric power series
\bean \label{hypereomfun} 
f_i(\hbar,\uu,q,z) = \sum\limits_{d=0}^{\infty}\, \Big( \prod\limits_{j=1}^{n}\, \dfrac{(u_i \hbar / u_j)_{d}}{(u_i q / u_j)_{d}}\Big) \, z^{d}, 
\eean 
where $(x,q)_d:=(1-x) (1-x q)\dots (1-x q^{d-1})$.
Set
$
e_i(z) = \exp\Big({\frac{\ln(z) \ln(u_i)}{\ln(q)}}\Big),
$ so that $e_i(z q)= e_i(z) u_i$. 
One verifies that the functions
\bean \label{solbasis}
{F}_i(\hbar,\uu,q,z)=e_i(z) f_i(\hbar,\uu,q,z) 
\eean
satisfy (\ref{qdedef}) for all $i=1,\dots, n$. Thus, ${F}_i(\uu,\hbar,q,z)$ give a basis of solutions to (\ref{qdedef}).
\subsection{\label{remsect}} 
From (\ref{qdedef}) it is clear that the vectors
\bean \label{funcolumns}
{\Psi}_i(\hbar,\uu,q,z)=\left(\begin{array}{l}
{F}_{i}(\hbar,\uu,q,z)\\
{F}_{i}(\hbar,\uu,q,z q)\\
\dots\\
{F}_{i}(\hbar,\uu,q,z q^{n-1})
\end{array}\right)
\eean
satisfy a first order $q$-difference equation
\bean \label{qde}
{\Psi}_i(\hbar,\uu,q,z q) = \Mop(\hbar,\uu,q,z)\, {\Psi}_i(\hbar,\uu,q,z)
\eean
for all $i=1,\dots, n$, where $ \Mop(\hbar,\uu,q,z)$ is the companion matrix for the characteristic polynomial of
(\ref{qdeoper}):
$$
\Mop(\hbar,\uu,q,z)=\left(\begin{array}{cccccc}
0 &1 &0&0&\dots&0\\
0&0&1&0&\dots&0\\
0&0&0&1&\dots&0\\
..&..&..&..&..&..\\
 \frac{\bar{\alpha}_n(z-1)}{1-\hbar^{n} z} &\frac{\bar{\alpha}_{n-1}(z \hbar-1)}{1-\hbar^{n} z}&\frac{\bar{\alpha}_{n-2}(z \hbar^2-1)}{\hbar^{n} z-1}&\frac{\bar{\alpha}_{n-3}(z \hbar^3-1)}{1-\hbar^{n} z}&\dots&\frac{\bar{\alpha}_1(z \hbar^{n-1}-1)}{1-\hbar^{n} z}\\
\end{array}\right)
$$
where $\bar{\alpha}_k$ denotes the coefficients of $x^k$ in $\prod_{i=1}^{n}(1-x u_i)$.  The matrix 
$$
{\Psi}(\hbar,\uu,q,z)=\Big({\Psi}_1(\hbar,\uu,q,z),\dots, {\Psi}_{n}(\hbar,\uu,q,z)\Big).
$$
with $i$-th column given by ${\Psi}_i(\uu,\hbar,q,z)$ is a fundamental solution matrix of the $q$-hypergeometric system (\ref{qde}). Explicitly, it has the following matrix elements
\bean \label{fundsolsmat}
{\Psi}_{i,j}(\hbar,\uu,q,z)={F}_j(\hbar,\uu,q,z q^{i-1}).
\eean

Any matrix with columns given by   $\mathbb{C}$ - linear combinations of the  columns of ${\Psi}(\uu,\hbar,q,z)$ produces a new fundamental solution matrix. Thus, other (analytic) fundamental solution matrices are of the form
$$
{\Psi}(\hbar,\uu,q,z) \Lambda, \ \ \Lambda \in \mathsf{GL}_n(\mathsf{\mathbb{C}}).
$$
Finally, we also note that the fundamental solution matrix has the form
\bean \label{analitfund}
{\Psi}(\hbar,\uu,q,z) = {\Psi}'(\hbar,\uu,q,z) \textsf{E}(z). 
\eean
where $\textsf{E}(z)$ is the diagonal matrix $\textsf{E}(z)=\diag(e_1(z),e_2(z),\dots, e_n(z))$, and ${\Psi}'(\uu,\hbar,q,z)$ is analytic near $z=0$, i.e., is a power series in $z$.

\subsection{}
Throughout this paper we will be using both, the multiplicative parameters $\uu=(u_1,\dots, u_n)$, $\hbar$  and additive parameters $\aa=(a_1,\dots, a_n)$, $h$ which are related by
\bean  \label{multoadd}
u_1 = q^{a_1}, \dots, u_n =q^{a_n}, \ \ \hbar=q^{h}.
\eean
In the additive notations we write: 
\bean \label{explfundam}
\Psi_{i,j}(h,\aa,q,z)= f_{i}(h,\aa,q,z q^{i-1}) z^{a_j} q^{a_j(i-1)}
\eean
where we denote by $\Psi_{i,j}(h,\aa,q,z)$ and $f_{i}(h,\aa,q,z)$ the functions (\ref{fundsolsmat}) and (\ref{hypereomfun}) with parameters (\ref{multoadd}) respectively.

\section{Vertex function of $X=T^{*}\mathbb{P}^{n-1}$ \label{versec}}
 QDE (\ref{qde}) describes the equivariant quantum $K$-theory of the cotangent bundles over projectile spaces \cite{OS22}. We briefly overview the connection in this section. 
\subsection{\label{ktheorydef}}
Let $(\matC^{\times})^{n}$ be the torus acting on $\matC^{n}$ by scaling the coordinate lines with characters $u_1,\dots, u_n$. We have an induced action of this torus on the cotangent bundle over projectile space $X=T^{*} \mathbb{P}(\matC^n)$.  Let $T=(\matC^{\times})^{n} \times \matC^{\times}$ be a bigger torus acting on $X$: the first factor $(\matC^{\times})^{n}$ acts as before and the second $\matC^{\times}$ scales the cotangent directions with a character $\hbar^{-1}$. The $T$-equivariant $K$-theory of $X$ is the ring:
\bean \label{ktheorC}
K_{T}(X) = \matZ[L^{\pm},u_1^{\pm},\dots, u^{\pm}_n, \hbar^{\pm}]/(L-u_1) \dots (L-u_n).
\eean
where $L$ is the class of the line bundle $\mathcal{O}(1)$. Abusing notations we denote by the same symbols $u_1,\dots,u_n$ the trivial equivariant line bundles associated to the characters $u_i$. In this language, the additive parameters 
$a_i$, $h$ may be thought of as the first Chern classes of the line bundles $u_i$ and $\hbar$. Similarly, let $x$ be the first Chern class of $L$, then in our conventions $L=q^{x}$.

\subsection{} 
Let us define a formal extension of this ring
$$
\hat{K}_{T}(X) = \mathrm{completion \ of} \ \  K_{T}(X)[z^x][[q]],
$$
so that it contains the following gamma class:
\bean \label{gammaclass}
\Phi(h,\aa,x,q,z)=\Big(\dfrac{(q^{h},q)_{\infty}}{(q,q)_{\infty}}\Big)^{\!\!n} \, \prod\limits_{i=1}^{n} \dfrac{(q^{\,x+1-a_i}  ,q )_{\infty}}{(q^{\,x+h-a_i}  ,q )_{\infty}}\,  z^{\,x}  \in \hat{K}_{T}(X),
\eean
where
$(u,q)_{\infty}=\prod\limits_{m=0}^{\infty}\, (1- u q^{m})$
is the reciprocal of the $q$ - gamma function. 
The normalized vertex function of $X$ is a generating function counting equivariant quasimaps to quiver varieties, see Section 7.2 of \cite{Oko17} for definitions. For $X$ the vertex function has the following form: 
\bean \label{normvert}
\widetilde{V}(h,\aa,x,q,z) =  \sum\limits_{d=0}^{\infty}\, \Phi(h,\aa,x+d,q,z) \, \in \hat{K}_{T}(X)[[z]]
\eean
One verifies that 
$$
P(\uu,\hbar,z) \, \widetilde{V}(h,\aa,x,q,z) = (L-u_1)\dots (L-u_n) \, \Phi(h,\aa,x,q,z)
$$
where $P(\uu,\hbar,z)$ is the $q$-difference operator given by (\ref{qdeoper}). 
Note that in (\ref{ktheorC}) the vertex function  solves the $q$-hypergeometric equation (\ref{qdedef}). 

\subsection{}
The components of the class (\ref{normvert}) in any chosen basis of $K$-theory (\ref{ktheorC}) provide a basis of solutions of the $q$-hypergeometric system (\ref{qdedef}). One convenient basis consists of the $K$-theory classes of the  $T$-fixed points $X^{T}$. The components of the class (\ref{normvert}) in this basis are obtained 
by the specialization $L=u_i$ or, in the additive notations, by the specialization $x=a_i$. In this way we obtain the basis of solutions given, up to a normalization multiple, by the hypergeometric series (\ref{solbasis}):
\bean \label{normhyperg}
\widetilde{V}(h,\aa,a_i,q,z) =  \Phi(h,\aa,a_i,q,z) {F}_i(h,\aa,q,z).
\eean
In this way, the vertex function (\ref{normvert}) provides a uniform ``basis free'' description of the solutions to (\ref{qdedef}).

\subsection{}
It will be convenient to normalize all solutions as in the vertex  (\ref{normhyperg}). In particular, we introduce the normalized fundamental solution matrix 
$$
\tilde{\Psi}(h,\aa,q,z)=\Big(\tilde{\Psi}_1(h,\aa,q,z),\dots, \tilde{\Psi}_{n}(h,\aa,q,z)\Big),
$$
with $i$-th column 
$$
\tilde{\Psi}_i(h,\aa,q,z) = {\Psi}'_i(h,\aa,q,z)  \Phi(h,\aa,a_i,q,z)
$$
where ${\Psi}'_i(h,\aa,q,z)$ is the analytic part of the fundamental solution matrix as defined in  (\ref{analitfund}). We can also write the last formula simply as
\bean \label{psitildef}
\tilde{\Psi}(h,\aa,q,z) = {\Psi}'(h,\aa,q,z) \Phi(h,\aa,x,q,z)
\eean
where $\Phi(\aa,h,x,q,z)$ is now understood as the operator of multiplication by the gamma class (\ref{gammaclass}) in the equivariant $K$ - theory. In particular, in the basis of the $T$-fixed points this operator is represented by the diagonal matrix with eigenvalues given by $\Phi(h,\aa,a_i,q,z)$, $i=1,\dots, n$. 

\subsection{}
The normalized fundamental solution matrix $\tilde{\Psi}(\aa,h,q,z)$  has an enumerative meaning: it coincides with the {\it capping operator} for $X$, which is the partition function of the relative quasimaps to $X$, see Section 7.4 in \cite{Oko17} for the definitions.  In particular, normalized as in (\ref{psitildef}) the fundamental solution matrix satisfies the {\it shift equations} which are the difference equations in the equivariant parameters $\aa$ and $h$. These equations exist in general for an arbitrary Nakajima variety. Applied to $X=T^{*}\mathbb{P}^{n-1}$ they can be formulated as follows:
\begin{thm}[Theorem 8.2.20, \cite{Oko17}] \label{shift_thm}
The normalized fundamental solution matrix (\ref{psitildef}) satisfies the following system of difference equations
\bean \nonumber
\tilde{\Psi}(h,a_1,\dots,a_i +1, \dots a_n,q,z)  = A_i(\aa,h,q,z) \tilde{\Psi}(h,a_1,\dots,a_i, \dots a_n,q,z),\\  \nonumber
\tilde{\Psi}(h+1,a_1,\dots,a_i, \dots a_n,q,z)  = H(\aa,h,q,z) \tilde{\Psi}(h,a_1,\dots,a_i, \dots a_n,q,z).
\eean
where $A_i(\aa,h,q,z)$ and $H(\aa,h,q,z)$ denote matrices whose entries are rational functions in the equivariant parameters $\uu$, $\hbar$, $q$ and $z$:
$
A_i(\aa,h,q,z), H(\aa,h,q,z)  \in \mathsf{GL}_{n}(\matQ(\uu,\hbar,z,q) ).
$
\end{thm}

These difference equations are useful for the analysis of QDEs, particularly when the parameters $h$, $\aa$ are specialized to $\mathbb{Z}_p$-values, to which we now switch out attention.

\section{$p$ - adic completions and norms \label{gaucomp}}
In this section, we provide a brief overview of 
$p$-adic analysis, 
$p$-adic analytic functions, and give several illustrative examples of simple Frobenius structures.
Our running example in this section is the quantum differential and quantum difference equation associated with a ``point'' $X=T^{*} \mathbb{P}^0$. The Frobenius structures arising in this case can be described using only elementary combinatorics.

\subsection{} 
We denote by $|\cdot|_p$ the $p$-adic norm on $\mathbb{Q}$ normalized so that $|p^s|=1/p^s$.  The field of $p$-adic numbers $\mathbb{Q}_p$ is the completion of $\mathbb{Q}$ with respect to this norm. $\mathbb{Z}_p\subset \mathbb{Q}_p$ denotes the set of integers of $\mathbb{Q}_p$ which is the subset of elements with norm $|x|_p\leq 1$. We also denote by $\Omega$ the completion of the algebraic closure of $\mathbb{Q}_p$.

Let $f(z)$ be a rational function:
$$
f(z) = \dfrac{\sum_{i} a_i z^i}{\sum_{j} b_j z^j} \in \mathbb{Q}(z)
$$
The $p$-adic Gauss norm on $\mathbb{Q}(z)$ is defined by
$$
|f(z)|_{\rm Gauss} := \dfrac{\textrm{max}_i\{|a_i|_p\}}{\textrm{max}_j\{|b_i|_p\}}
$$
We denote by  $E_p$ the field of $p$-adic analytic functions which is the completion of $\mathbb{Q}(z)$ with respect to the Gauss norm. By definition, an  element $f(z) \in E_p$ is a sum
$$
f(z) = \sum\limits_{n=0}^{\infty}\, f_i(z), \ \ f_i(z) \in \mathbb{Q}(z)
$$
converging in the Gauss norm. Note that 
$
|f(z)|_{\rm Gauss} = \max\{|f_i(z)|_{\rm Gauss}:i=0,1,\dots\}. 
$

\subsection{} 
One of the main questions we encounter in this paper is the following: let $f(z) \in \mathbb{Q}[[z]]$ be a power series. When is $f(z)$ a Taylor expansion of an element $g(z) \in E_p$? 

Since the norm of analytic elements is bounded, there exist $n$ such that $|p^n g(z)|_{\rm Gauss}\leq 1$.  Thus, if $f(z)$ is the series expansion of $g(z) \in E_p$ then we can ``clear denominators'' - the coefficients of power series $p^n f(z)$ do not have powers of $p$ appearing in the denominators.  In particular, the coefficients of the power series $p^n f(z)$ have well defined reductions $\pmod{p^s}$ for $s=1,2,\dots$. These reductions can be used to determine whether $
f(z)$ is the expansion of an analytic function using the following elementary lemma.

\begin{lem} \label{lemmagaussnorm}
A power series $f(z) \in \mathbb{Q}[[z]]$ is a Taylor expansion of an analytic element $g(z) \in E_p$ if and only if there exist $n \in \mathbb{N}$ such that for any $s\in \mathbb{N}$
\bean \label{condit}
p^n f(z) \equiv g_s(z) \pmod{p^s}
\eean 
where $g_s(z)$ is a power series expansion of some rational function from $\mathbb{Q}(z)$. 
\end{lem}
\begin{proof}
First assume $f(z) \in \mathbb{Q}[[z]]$ is such that (\ref{condit}) holds for some $n$. 
Let us define $f_i(z)$ by
$$
f_{i}(z)= g_{i+1}(z)-g_{i}(z)
$$
By construction $|f_{i}(z)|_p <1/p^{i}$, thus the sum $\sum_{i=0}^{\infty} f_{i}(z)$ converges in the Gauss norm and we obtain:
$$
f(z) = \dfrac{1}{p^n} \sum_{i=0}^{\infty} f_{i}(z) \in E_p
$$
For the opposite direction, assume $f(z)$ is a power series in $z$ which is a Taylor expansion of some ${g}(z) \in E_p$. We have
\bean \label{gsum}
{g}(z) =\sum\limits_{i=0}^{\infty} \, g_i(z)
\eean 
for some $g_i(z) \in \mathbb{Q}(z)$. Assume that $|{g}(z)|_{\rm Gauss} =p^{n}$ for some $n \in \mathbb{Z}$. 
Since (\ref{gsum}) converges in the Gauss norm it follows that
for any $s \in \mathbb{N}$ there is $N_s \in \mathbb{N}$ such that $p^n g_{i}(z) \equiv 0 \pmod{p^s}$ for all $i>N_s$. Therefore,
$$
p^n f(z) \equiv  \sum\limits_{n=0}^{N_s}\, g_i(z) \pmod{p^s} 
$$
\end{proof}
In Section \ref{numsec} we will demonstrate how this Lemma can be used to discover non-trivial Frobenius structures numerically.

\vspace{5mm}

\noindent
{\bf Example:} Let $h \in \mathbb{Q}$ be a rational number such that $h \in \mathbb{Z}_p$.  Let us consider a power series 
\bean \label{hyper0ex}
f(h, z) = \sum\limits_{n=0}^{\infty}\, (-1)^n \binom{h}{n} z^n
\eean 
which is a Taylor expansion of $(1-z)^{h}$. It is known that this series is not an expansion of an element from $E_p$ for generic $h$. However, the power series 
\bean \label{hyperint0}
\Fr(h, z)=\dfrac{f(p h,z)}{f(h, z^p)} 
\eean
is in $E_p$.  Indeed
$$
\Fr(h,z) = \dfrac{(1-z)^{p h}}{(1-z^p)^{h}} =  \Big(1+ \dfrac{(1-z)^p-(1-z^p)}{1-z^p} \Big)^{h}
$$
Note that $(1-z)^p-(1-z^p)=p \alpha(z)$ for some $\alpha(z)\in \mathbb{Z}[z]$ and therefore 
$$
\Fr(h,z) = \sum\limits_{n=0}^{\infty}\, \binom{h}{n} p^n \dfrac{\alpha(z)^n}{(1-z^p)^n}
$$
For $h \in \mathbb{Z}_p$ we also have $\binom{h}{n} \in \mathbb{Z}_p$ for any $n$.  From which we see that:
\bean \label{approxu}
\Fr(h,z) \equiv  \sum\limits_{n=0}^{s-1}\,\binom{h}{n} p^n \dfrac{\alpha(z)^n}{(1-z^p)^n} \pmod{p^s}
\eean 
Therefore $\Fr(h,z) \in E_p$ by the previous Lemma.

\subsection{}
Let $V(\alpha,z)$ denote the finite dimensional $\mathbb{Q}_p$-space of solutions to some $p$-adic differential equation in $z$ with parameters $\alpha$.  As we explain in the next section, a Frobenius intertwiner gives a map:
$$
V(\alpha, z) \stackrel{\Fr(z)}{\longrightarrow} V(\alpha', z^p)
$$
As a map between two {\it different} vector spaces $\Fr(z)$ carries no interesting information. However, when $\alpha'=\alpha$ at the points $z^p=z$ it becomes an automorphism of $V(\alpha,z)$ and the characteristic polynomial of $\Fr(z)$ at such points is an interesting invariants of the QDEs.  The invariants appearing in this way find many applications in the theory of Gauss and exponential sums \cite{Dw74,Sp80,K90}.

The points $z\in \mathbb{Q}_p$ satisfying $z^p=z$ are the Teichmüller elements:
recall that for any $a\in \mathbb{F}_p$ there is unique $[a]\in \mathbb{Q}_p$ such that
$[a]^p=[a]$ and $[a] \equiv a \pmod{p}$.
Clearly, the map $\mathbb{F}_p^{\times} \to \mathbb{Q}^{\times}_p$, sending $a$ to $[a]$, is a multiplicative character. The element $[a]\in \mathbb{Q}_p$ is called the Teichmüller representative or the Teichmüller lift of $a$.

By definition, if $a\neq0$ then $|[a]|_p =1$. For this reason, a power series $f(z) \in \mathbb{Q}[[z]]$ usually can not be evaluated at such points, since the convergence is not guaranteed. A notable property of the analytic elements $g(z) \in E_p$ is that they often can be evaluated at points $z\in \mathbb{Z}_p$ with $|z|_p=1$: 
 if $f(z) \in E_p$ then up to terms with norm $1/p^s$ it can be approximated by  a rational function. The values of these rational functions at $[a]$ then approximate $f([a])$ with an error of norm $1/p^s$.

For these reasons, Frobenius intertwiners are naturally required to have coefficients in $E_p$.  This motivates the definition we give in the next section.  

\vspace{5mm}

\noindent
{\bf Example:}  
Continuing our running example, we note that the power series (\ref{hyper0ex}) spans $1$-dimensional space of solutions of the simplest 
hypergeometric differential equation
\bean \label{interhyper}
(1-z) z \dfrac{d f(z)}{dz} + z h f(z)=0 
\eean 
In the hypergeometric notations $f(z)={}_0F_{1}(-h;;z)$. By its definition, multiplication by power series (\ref{hyperint0}) maps solutions with parameters $(h,z^p)$ to solutions with parameters $(ph,z)$. As we have already seen, $U(h,z)$ is in $E_p$. Thus, $U(h,z)$ is  a Frobenius intertwiner between these solutions.  

Assume that $h$ is such that $(p-1)h \in \mathbb{Z}$. Then
$(1-z)^{h(1-p)}$ is a rational function and therefore the composition
\bean \label{interprime}
\Fr'(z)= (1-z)^{h(1-p)} \Fr(h,z) = \dfrac{(1-z)^h}{(1-z^p)^h}
\eean
is also in $E_p$. Thus $\Fr'(z)$ is a Frobenius intertwiner between solutions at $(h,z^p)$ and $(h,z)$. Using the approximations by rational functions (\ref{approxu}) one can prove that the value of this intertwiner at the Teichmüller elements equals:
\bean \label{punroot}
\Fr'([a]) = [a-1]^{h(1-p)}
\eean
Note that a naive substitution of $z=[a]$ to (\ref{interprime}) would give an incorrect answer $\Fr'([a])=1$ (since $[a]^p=[a]$). The $p$-adic unit root (\ref{punroot}) is the ``interesting invariant'' of the differential equation (\ref{interhyper}) which we mentioned above.

\subsection{} 
Let show that $\Fr(h, z)$ given by (\ref{hyperint0}) is in $E_p$ in a different way using the $q$-deformation. 
First, let us fix $q\in \Omega$ of the form $q=1+t$ with $|t|_p<1$. In this case, for $h\in \mathbb{Z}_p$ we have well defined elements  $q^{\pm h}  \in \Omega$. The $q$-deformation of the hypergeometric $f(z)={}_0F_{1}(-h;;z)$ is given by the power series
\bean \label{qdefpow0}
f(h, q,z)= \sum\limits_{n=0}^{\infty}\, \dfrac{(q^{-h},q)_n}{(q,q)_n} z^n, \ \ \ (x,q)_{n}=(1-x) (1-x q)\dots (1-x q^{n-1}),
\eean
which solves the $q$-analog of (\ref{interhyper}):
\bean \label{qdee0}
f(h, q,z q) = \dfrac{(1-z)}{(1-z q^{-h})} f(h, q,z)  
\eean
At $h\in \mathbb{N}$ this power series truncates to a Laurent polynomial
\bean \label{fintpoint}
f(h, q,z) = (1-z q^{-1}) (1-z q^{-2}) \dots (1-z q^{-h}). 
\eean 
Assume that $\zeta\neq 1$ is a unit root $\zeta^{p^s}=1$. Assume further that the order of the root $\zeta$ is $p^m$ for some $m\leq s$. Let $h \in \mathbb{N}$ and write $h=h^{(s)} + p^{s-1} h'$  for some natural numbers $h^{(s)}$ and $h'$. 
From (\ref{fintpoint}) we find 
$$
f(p h, \zeta,z) = f(p h^{(s)}, \zeta,z) (1-z^{p^m})^{h' p^{s-m}}   
$$
Since $\zeta^p$ is also a unit root or order $p^{m-1}$ we obtain similarly
$$
f(h,\zeta^p,z^p) = f(h^{(s)}, \zeta^p,z^p) (1-z^{p^m})^{h' p^{s-m}}  
$$
where in the last two equalities we used that
$
\prod_{i=0}^{m-1}(1-z \zeta^i) = (1-z^m)
$
if $\zeta$ has  order~$m$.

Let us consider the power series
$$
\Fr(h,q,z) = \dfrac{f(p h, q,z) }{f(h, q^p,z^p)}
$$
From the above computations, we see that for $h\in \mathbb{N}$ we have
\bean \label{interatroot}
\Fr(h,\zeta,z) =  \Fr(h^{(s)},\zeta,z)  
\eean
where $h^{(s)}$ is a natural number defined by $h^{(s)} = h \pmod{p^{s-1}}$. 

The equation (\ref{interatroot}) is crucial. Note that (\ref{interatroot}) extends to $p$-adic integers $h\in \mathbb{Z}_p$. Clearly, unit roots $q^{p^s}=1$ such that $q\neq 1$ are the roots of the polynomial:
$$
[p^s]_q  = \dfrac{1-q^{p^s}}{1-q}
$$
Thus, the relation (\ref{interatroot}) implies that for any $h\in \mathbb{Z}_p$ we have
$$
\Fr(h,q,z) =  \Fr(h^{(s)},q,z)  + [p^s]_p \, \Big( \ \ \dots \ \ \Big)
$$
where $\Fr(h^{(s)},q,z)$ is a rational function (since $h^{(s)}$ is integral) and $\dots$ stands for a power series in $z$ with coefficients regular at $p^s$-th roots. Note that for $q=1+t$ with $|t|_p<1$ we have $[p^l]_p = 0 \pmod {p^s}$ for sufficiently large $l$ thus
\bean \label{qfrobrel}
\Fr(h,q,z)  \pmod {p^s}.
\eean 
is a rational function of $z$.

In the limit $q\to 1$ $\Fr(h,q,z)$ converges to  (\ref{hyperint0}) and therefore 
$
\Fr(h,z) \pmod {p^s}.
$
is also rational and by Lemma \ref{lemmagaussnorm} we conclude that $\Fr(h,z) \in E_p$. 

The above computation (\ref{qfrobrel})
 also shows what $\Fr(h,q,z)\in \hat{E}_p$ where 
$\hat{E}_p$ is the completion of $\Omega(z)$ with respect to Gauss norm (we need to replace $\mathbb{Q}$ with $\Omega$ to accommodate for $q \in \Omega$). Thus, $\Fr(h,q,z)$ is a Frobenius intertwiner between $f(h,q^p,z^p)$ and $f(p h,q,z)$.

\vspace{3mm}

Finally, let us note that the hypergeometric differential equation (\ref{interhyper}) is the quantum differential equation associated with a zero-dimensional quiver variety $X=T^{*}\mathbb{P}^{0}$ (enumerative geometry of such varieties may be non-trivial, see \cite{DS}). 
The $q$-difference equation (\ref{qdee0}) is the $K$-theoretic quantum difference equation for this $X$. The power series (\ref{hyper0ex}) and (\ref{qdefpow0}) represent the ``fundamental solution matrices'' of these equations. The running example of this section is thus, the most basic $n=1$ case of the hypergeometric functions (\ref{hypereomfun}) we consider in this paper.

\section{Frobenius intertwiner for $q$-hypergeometric equation \label{frobintsect}}

\subsection{\label{frobdef}} 
Let us consider the  $q$-difference hypergeometric equation (\ref{qde}).
Unless otherwise specified, we assume throughout this section that $h\in \mathbb{Z}_p$, $\aa \in \mathbb{Z}^n_p$. We assume that $q\in \Omega$ is of the form $q=1+t$ with $|t|_p<1$. In this case 
$$
\hbar =q^{h}, \ \ u_1 =q^{a_1},\dots , u_n =q^{a_n}
$$
are well defined in $\Omega$. Let $\hat{E}_p$ be the completion of $\Omega(z)$ with respect to the Gauss norm. 
The following is motivated by discussion of Section \ref{gaucomp}. 
\begin{defn}
A Frobenius intertwiner from a $q$-hypergeometric system with parameters $(h,\aa,q^p,z^p)$ to a $q$-hypergeometric system with parameters $(h',\aa',q,z)$ is an element $\Fr(z) \in \mathsf{GL}_n(\hat{E}_p)$ such that
$ \Mop(\hbar,\aa,q^p,z^p) = \mathsf{U}(z q)^{-1} \,\Mop(\hbar',\aa',q,z)\,  \mathsf{U}(z)$. 
\end{defn}

If ${\Psi}(h,\aa,q^p,z^p)$ is a fundamental solution matrix for the $q$-hypergeometric system with parameters $(h,\aa,q^p,z^p)$ and $\Fr(h,\aa,q,z)$ is a Frobenius intertwiner as in the Definition above 
then $\Fr(h,\aa,q,z) {\Psi}(h,\aa,q^p,z^p)$ is a fundamental solution matrix for the $q$-hypergeometric system with parameters $(h',\aa',q, z)$. Thus, it is of the form ${\Psi}(h', \aa',q,z) \Lambda$ where $\Lambda \in \mathsf{GL}_n(\Omega)$ is some constant matrix. We obtain a relation 
\bean \label{frobstru}
{\Psi}(h',\aa',q,z) \Lambda = \Fr(h,\aa,q,z) {\Psi}(h,\aa,q^p,z^p).
\eean
$\Fr(h,\aa,q,0)$ is called the {\it constant term of the Frobenius intertwiner}. If the fundamental solution matrix is normalized by ${\Psi}(h,\aa,q,0)={\rm Id}$ (the identity matrix), then $\Fr(0) =\Lambda$.

\subsection{}
Let us consider the power series:
\bean \label{froint}
\Fr(h,\aa,q,z) = \tilde{\Psi}(p h,p\aa,q,z) \tilde{\Psi}( h,\aa,q^{p},z^p)^{-1} \in  K_{T}(X)^{\otimes 2}[[z]]
\eean 
where $ \tilde{\Psi}(h,\aa,q,z)$ denotes the fundamental solution matrix normalized as in (\ref{psitildef}). Using the natural non-degenerate paring on $K_{T}(X)$ we can also view it as an operator 
$$
\Fr(h,\aa,q,z) :  K_{T}(X)[[z]] \longrightarrow  K_{T}(X)[[z]]
$$
In particular, the constant term is a linear map in $K$-theory 
\bean \label{constterm2}
\Fr(h,\aa,q,0) :  K_{T}(X) \longrightarrow  K_{T}(X)
\eean 

\vspace{5mm}

\noindent
The goal of this section is to prove the following result (see Section \ref{ktheorydef} for notations and description of the ring $K_{T}(X)$):
\begin{thm}
\label{mainthm} The power series $\Fr(h,\aa,q,z)$ 
defined by (\ref{froint}) is a Frobenius intertwiner between the $q$-hypergeometric systems with parameters  $(h,\aa,q^p,z^p)$ and $(p h,p \aa,q,z)$, where $p\aa= (p a_1,\dots, p a_n)$. The constant term of the intertwiner (\ref{constterm2}) acts in $K_{T}(X)$ as the operator of multiplication by the K-theory class 
\bean \label{frobzeroterm}
\Fr(h,\aa,q, 0) = \prod\limits_{j=1}^{n} \, \dfrac{\Gamma_{p,q}(p x+ p h - p a_j )}{\Gamma_{p,q}(p x-p a_j ) \Gamma_{p,q}(p h )}.
\eean
\end{thm} 
Definition of  $p$-adic $q$-Gamma function $\Gamma_{p,q}$ is reminded in Section \ref{pqgamm}.

\subsection{} 
 Let us first analyze the  constant term of this operator.
\begin{prop} The constant term of the power series (\ref{froint}) acts in $K_{T}(X)$ as multiplication by the following K-theory class:  
\bean \label{zertermkob}
\Fr(h,\aa,q,0)  = \prod\limits_{j=1}^{n} \, \dfrac{\Gamma_{p,q}(p x+ p h - p a_j )}{\Gamma_{p,q}(p x-p a_j ) \Gamma_{p,q}(p h )} 
\eean 
where $\Gamma_{p,q}$ denotes the $p$-adic $q$-gamma function . 
\end{prop}
\begin{proof}
From the normalization (\ref{psitildef}) we find:
$$
\Fr(h,\aa,q,0) = \dfrac{\Phi(p h,p\aa,p x,q,z)}{\Phi(h,\aa,x,q^p,z^p)} 
$$
From  definition (\ref{gammaclass}) we see that this ratio is exactly the product of factors as in (\ref{qfactors}). The Proposition follows by Lemma \ref{extlem}.
\end{proof}
If  the set of the $T$-fixed points $X^{T}$ is finite, then multiplication by any equivariant $K$-theory class is diagonal in the basis of $K_{T}(X)$ given by the classes of $T$-fixed points. In our case $X=T^{*}\matP^{n-1}$ the set $X^{T}$ is  finite  when all $a_1,\dots, a_n$ are pairwise distinct. The corresponding eigenvalues of multiplication by $\Fr(h,\aa,q,0)$ are obtained by specialization $x=a_i$. We thus conclude with the following result. 
\begin{cor} \label{fpcorrol}
If  $a_1,\dots, a_n$ are pairwise distinct then the constant term of the Frobenius intertwiner $\Fr(h,\aa,q,0)$ is diagonal in the basis of torus fixed points of $K_{T}(X)$. The corresponding eigenvalues are equal:
\bean \label{frobconj}
\Fr(h,\aa,q,0)_{i,i}=\prod\limits_{j=1}^{n}\, \dfrac{\Gamma_{p,q}(p(a_i-a_j +h) )}{\Gamma_{p,q}(p(a_i-a_j)) \Gamma_{p,q}(p h )}, \ \ \ i=1,\dots,n.   
\eean 
\end{cor}

This result has the following $K$-theoretic formulation. Let $\mathcal{V}$ be a rank $r$ vector bundle over $X$. Using the splitting principle, we may assume that in $K$-theory $\mathcal{V}$ is a sum of $r$ line bundles with Chern roots $x_1,\dots, x_r$. We can thus define a $K$-theory valued characteristic class by
\bean \label{KthGclass}
\Gamma_{p,q}(\mathcal{V}) = \prod_{i=1}^{r} \Gamma_{p,q}(x_i)^{-1}
\eean
Recall that every Nakajima variety is equipped with a well-defined polarization bundle $T^{1/2} X$, i.e., $T^{1/2} X$  is a $K$-theory class representing a ``half of the tangent bundle'' \cite{MO19}:
$$
TX  =T^{1/2} X +\hbar^{-1} \overline{T^{1/2} X} 
$$
where $\hbar$ is the $T$-character of the symplectic form. In our case $X=T^{*}\mathbb{P}^{n-1}$, the polarization can be chosen in the form 
$T^{1/2} X=W^{*} \otimes L - 1$. Where $W$ is a trivial rank $r$ vector bundle with Chern roots given by the equivariant parameters $a_1,\dots,a_n$,  and  $L$ is the tautological line bundle with Chern root $x$. We note that the arguments of the gamma functions in (\ref{frobzeroterm}) is over the $p$-multiples of the Chern roots of the virtual bundle $(q-\hbar) T^{1/2}X$, in other words is over the Chern roots of $\psi^p( (q-\hbar) T^{1/2}X)$ where $\psi^p$ is the $p$-th Adams operation in the equivariant $K$-theory (which acts on the Chern roots by $x_i\to p x_i$). Combining all this together we obtain: 
\begin{prop}  \label{propKthCT}
The constant term of (\ref{constterm2}) acts as multiplication  by the K-theory class
\bean 
\Fr(h,\aa,q,0)  = \Gamma_{p,q}((q^p-\hbar^p) T^{1/2} X^{(p)})
\eean 
where $T^{1/2} X^{(p)}$ is  the polarization bundle of $X$ twisted by the $p$-th Adams operation.  
\end{prop}
We expect that in this form the Proposition holds for every Nakajima variety $X$.

\subsection{}
Let us consider the vertex function (\ref{normvert}) normalized by a factor $\Phi(\aa,h,x,q,z)$  given by (\ref{gammaclass}).  If $h\in \mathbb{N}$, using the notations as in (\ref{poham1}), (\ref{poham2}) we can write the vertex function in the form
$$
\widetilde{V}(h,\aa,x, q,z) = z^{x} \,\sum\limits_{d=0}^{\infty}\, \Big(\prod\limits_{i=1}^{n} \dfrac{[q^{\,x+d-a_i},q]_{h}}{[1,q]_{h}} \Big) z^{d}. 
$$

\begin{prop} \label{verconthm}
Let $\zeta\neq 1$ be a unit root $\zeta^{p^s}=1$ and let $h' \in \mathbb{N}$ then we have the following identity for the vertex functions 
$$
\left.\tilde{V}(h' p^s, p \aa, px,  q,z)\right|_{q=\zeta} =  \left.\tilde{V}(h' p^{s-1}, \aa, x,  q^p,z^p)\right|_{q=\zeta}
$$
\end{prop}

\begin{proof}
We have:
 $$
 \left.\tilde{V}( h' p^s, p \aa, px, q,z)\right|_{q=\zeta}= z^{p x} \, \sum\limits_{d=0}^{\infty}\, \left( 
\prod\limits_{i=1}^{n} \left.\dfrac{[q^{p x - p a_i +d},q]_{h' p^s}}{[1,q]_{h' p^s}}\right)\right|_{q=\zeta} z^{d}
 $$
By Lemma \ref{coeffcongr} the coefficients of this sum vanish unless $p\mid d$, therefore
we write
$$
 \left.\tilde{V}(h' p^s, p \aa, px, q,z)\right|_{q=\zeta}= z^{p x} \, \sum\limits_{d=0}^{\infty}\, \left( 
\prod\limits_{i=1}^{n} \left.\dfrac{[q^{p x - p a_i +p d},q]_{h' p^s}}{[1,q]_{h' p^s}}\right)\right|_{q=\zeta} z^{p d}
$$
again, by Lemma \ref{coeffcongr} we obtain
$$
 z^{p x} \,\sum\limits_{d=0}^{\infty}\, \left.\left( 
\prod\limits_{i=1}^{n} \dfrac{[q^{p x - p a_i +p d},q]_{h' p^s}}{[1,q]_{h' p^s}}\right)\right|_{q=\zeta} z^{p d} =z^{p x} \, \sum\limits_{d=0}^{\infty}\,\left.\left( 
\prod\limits_{i=1}^{n} \dfrac{[q^{p x - p a_i +p d},q^p]_{h' p^{s-1}}}{[1,q^p]_{h' p^{s-1}}}\right)\right|_{q=\zeta} z^{p d} 
$$
and the right side is $\left.\tilde{V}(  h' p^{s-1},\aa, x, q^p,z^p)\right|_{q=\zeta}$. 
\end{proof}
Since the $K$-theory components of the vertex functions provide a basis in the space of solutions for the quantum difference equation, the above Proposition extends to the fundamental solution matrix which is normalized as the vertex function (\ref{normvert}), i.e., the fundamental solution matrix (\ref{psitildef}).  Thus, we have:
\begin{thm} \label{trivalizthm} The ratio of fundamental solution matrices (\ref{froint})
$$
\widetilde{\Psi}( h' p^s,p \aa, q, z) \widetilde{\Psi}(h'  p^{s-1}, \aa, q^{p}, z^p)^{-1}
$$
is well defined at $q$ given by  unit roots $\zeta\neq 1$, $\zeta^{p^s}=1$ and we have
\bean \label{psicongr}
\left.[\widetilde{\Psi}(h' p^s,p \aa,  q, z) \widetilde{\Psi}(h'  p^{s-1}, \aa, q^{p}, z^p)^{-1}\right]_{q=\zeta} =1.
\eean 
\end{thm}
\begin{proof}
For $h'\in \mathbb{N}$ the Theorem follows from the Previous proposition.  Since the right side of  (\ref{psicongr}) is independent of $h'$ and $\mathbb{N}$ is dense in $\mathbb{Z}_p$, this identity extends to  $h'\in\mathbb{Z}_p$.
\end{proof}
\subsection{} 
We note that the proof of Theorem \ref{trivalizthm}  uses only the fact that the $K$-theory components of the vertex function generate the whole  $q$-holonomic module, see for instance Theorem 4 in \cite{AO21}.  For this reason, this argument applies to any choice of the equivariant parameters including the most degenerate case $a_i=0$, which corresponds to the limit to the non-equivariant $K$-theory ring~$K(X)$.   

In the non-degenerate case $a_i\neq a_j$, $K_{T}(X)$ has a distinguished basis given by classes of the torus fixed points. In this basis 
the components of the fundamental solution matrix $\widetilde{\Psi}(p \aa, h' p^s, q, z)$ are given explicitly by $q$-hypergeometric series. We can use this presentation to prove Theorem \ref{trivalizthm}  using only elementary combinatorics. For future references, we outline such a proof in this section. 

Assume $a_i\neq a_j$.  First, from normalization (\ref{psitildef}) we have
$$
\widetilde{\Psi}(h' p^s,p \aa,  q, z) \widetilde{\Psi}(h' p^{s-1}, \aa,  q^{p},z^p)^{-1} = {\Psi}(h' p^s,p \aa,  q, z)\, \Fr(h' p^{s-1},\aa,q,0)\, {\Psi}( h' p^{s-1},\aa,  q^{p},z^p)^{-1}
$$
In the basis of $T$-fixed points in $K_{T}(X)$, ${\Psi}(\aa, h, q, z)$ has matrix elements (\ref{explfundam}). The middle term is the multiplication by $K$-theory class (\ref{zertermkob}). In the basis of $T$-fixed this operator is diagonal with the eigenvalues given by: 
\bean \label{midtermiden}
\Fr(h' p^{s-1},\aa,q,0)_{i,i} =\prod\limits_{j=1}^{n} \, \dfrac{\Gamma_{p,q}(p a_i+ h' p^s - p a_j )}{\Gamma_{p,q}(p a_i-p a_j ) \Gamma_{p,q}( h' p^s )} 
\eean 
At $q$ given by $p^s$-th unit roots $\Fr(h' p^{s-1},\aa,q,0)=1$ from Proposition \ref{uonecorr}.

Next, from (\ref{explfundam}), the matrix elements of the fundamental solution matrix have the form
\bean \label{pmat1}
{\Psi}_{i,j}(h' p^s, p \aa,  q, z) = f_{i}(h' p^s,p\aa, q,z q^{i-1}) z^{p a_j} q^{p a_j (i-1)}
\eean
and similarly 
\bean \label{pmat2}
{\Psi}_{i,j}(h' p^{s-1},\aa,  q^p, z^p) = f_{i}(h' p^{s-1},\aa, q^p,z^p q^{p(i-1)}) z^{p a_j} q^{p a_j (i-1)}
\eean

If $\zeta\neq 1$ and $\zeta^{p^s}=1$ we can assume that $\zeta$ is a {\it primitive} $p^m$-th unit root for some $1 \leq m\leq s$, then, by part b) of Lemma \ref{hypergeomlem} we obtain
$$
\left.f_{i}( h' p^s, p\aa,q,z q^{i-1})\right|_{q=\zeta}= \left.f_{i}( h' p^{s-m} p^{m},p\aa,q,z q^{i-1})\right|_{q=\zeta}  =  (1-z^{p^m})^{p^{s-m} h'} 
$$
Since $\zeta^p$ is also a primitive $p^{m-1}$-th unit root we obtain similarly: 
$$
\left.f_{i}(h' p^{s-1}, \aa, q^p,z^p q^{p(i-1)})\right|_{q=\zeta} = \left.f_{i}(h' p^{s-m} p^{m-1}, \aa, q^p,z^p q^{p(i-1)} )\right|_{q=\zeta} =   (1-z^{p^m})^{p^{s-m} h'}
$$
Therefore the matrices (\ref{pmat1}) and (\ref{pmat2}) are equal:
$$
\left.{\Psi}_{i,j}(h' p^s, p \aa,  q, z)\right|_{q=\zeta}= \left.{\Psi}_{i,j}(h' p^{s-1}, \aa,  q^p, z^p)\right|_{q=\zeta} = (1-z^{p^m})^{p^{s-m} h'} q^{p a_j (i-1)}
$$
Note also that for generic values $a_i\neq a_j$ these matrices are invertible and therefore 
$$
\left.[\widetilde{\Psi}( h' p^s, p \aa,q, z) \widetilde{\Psi}(h'  p^{s-1},  \aa,q^{p}, z^p)^{-1}\right]_{q=\zeta} =1.
$$

\subsection{\label{proofsection}} 
Here is the main result of this Section. 
\begin{thm} \label{newthm}
The power series $\Fr(h,\aa,q,z)$  given by (\ref{froint}) is well defined at $q$ given by unit roots $\zeta^{p^s}=1$. If $\zeta\neq 1$ is such unit root, then $\Fr(h,\aa,\zeta,z)$ is a Taylor series expansion of a rational  
function  from $\Omega(z,\hbar,u_1,\dots,u_n)$ where
\bean \label{mulpar}
\hbar = \zeta^h, \ \  u_1 = \zeta^{a_1},\dots, u_n =\zeta^{a_n}
\eean 
In particular, these rational functions satisfy the congruences 
\bean \label{qDwork}
\Fr(h,\aa,\zeta,z)  = \Fr(h^{(s)},\aa^{(s)},\zeta,z).
\eean
where $\aa^{(s)}=(a_1^{(s)},\dots, a^{(s)}_n)$ and $h^{(s)}$ are natural numbers defined  by $h^{(s)} = h \pmod {p^{s-1}}$, $a^{(s)}_i = a_i \pmod {p^{s-1}}$, $i=1,\dots, n$.


\end{thm}
\begin{proof}
Define $h^{(s)} \in \matN$ by $h^{(s)} = h \pmod{p^{s-1}}$. 
We have $h=h^{(s)} + p^{s-1} h'$ for some $h' \in \mathbb{Z}_p$. By  Theorem~\ref{shift_thm} we obtain:
$$
\widetilde{\Psi}(p h, p \aa,  q, z) = \widetilde{\Psi}( p h^{(s)} + h' p^{s}, p \aa, q, z)  = \left(\prod\limits_{i=h' p^{s}}^{p h^{(s)} + h' p^{s}-1} H(p \aa,i,q,z) \right)  \widetilde{\Psi}(h' p^{s}, p \aa,  q, z), 
$$
where $H(p \aa,i,q,z)$ is a certain invertible matrix whose coefficients are rational functions of~$z$ and (\ref{mulpar}).  Similarly, we have
$$
\widetilde{\Psi}(h, \aa,   q^p, z^p) = \left(\prod\limits_{i=h' p^{s-1}}^{h^{(s)} +h' p^{s-1} -1} H(\aa,i,q^p,z^p) \right)  \widetilde{\Psi}(h_1 p^{s-1}, \aa,  q^p, z^p)
$$
and therefore:
$$
\widetilde{\Psi}(p h, p \aa,  q, z)  \, \widetilde{\Psi}( h, \aa,  q^{p}, z^p)^{-1} = 
$$
$$
\left(\prod\limits_{i=h' p^{s}}^{p h^{(s)} + h' p^{s}-1} H(p \aa,i,q,z) \right)  \widetilde{\Psi}(h' p^{s}, p \aa,  q, z)  \widetilde{\Psi}( h' p^{s-1},\aa, q^p, z^p)^{-1} \left(\prod\limits_{i=h' p^{s-1}}^{h^{(s)} +h' p^{s-1} -1} H(\aa,i,q^p,z^p) \right)^{-1}
$$
By Theorem~\ref{trivalizthm} at $q=\zeta$ the middle term in the last expression is trivial. We note that $H(p \aa,i,\zeta,z)=H(p \aa,i \pmod{p^s},\zeta,z)$ since it depends on $i$ via $\zeta^{i} = \zeta^{i \pmod{p^s}}$. Thus, we obtain:
$$
\left[\widetilde{\Psi}(p h, p \aa,  q, z)  \, \widetilde{\Psi}(h, \aa,  q^{p}, z^p)^{-1}\right]_{q=\zeta} = \left(\prod\limits_{i=0}^{p h^{(s)} -1} H(p \aa,i,\zeta,z) \right) \left(\prod\limits_{i=0}^{h^{(s)} -1} H(\aa,i,\zeta^p,z^p) \right)^{-1} 
$$
which is a product of finitely many matrices whose entries are rational functions in $z$ and $\zeta^{p a_i}$. 
We have $\zeta^{ p a_i} = \zeta^{p a^{(s)}_i}$ where 
$a^{(s)}_i = a_i \mod p^{s-1}$.  This finishes the proof.  
\end{proof}

\begin{cor}
Let $\zeta\neq 1$ be a unit root satisfying $\zeta^{p^s}=1$, then we have
$$
\left.\Fr(h,\aa,q,z)\right|_{q=\zeta} =  \left(\prod\limits_{i=0}^{p h^{(s)} -1} H(p \aa^{(s)},i,\zeta,z) \right) \left(\prod\limits_{i=0}^{h^{(s)} -1} H(\aa^{(s)},i,\zeta^p,z^p) \right)^{-1}
$$
\end{cor}

\subsection{Proof of Theorem \ref{mainthm}}
We are now ready to proceed to the proof of Theorem \ref{mainthm}. 
\begin{proof}
By its definition (\ref{froint}), $\Fr(h,\aa,q, z)$ is map sending the fundamental solution $\Psi(h,\aa,q^p,z^p)$ to the fundamental solution $\Psi(p h,p\aa,q,z)$, i.e., it is an intertwiner between $q$-difference equations with these parameters. To show that it is a Frobenius intertwiner we need to check that $\Fr(h,\aa,q, z)\in \mathsf{GL}_n(\widehat{E}_p)$. 
We have
$$
\prod\limits_{{\zeta\neq 1,} \atop {\zeta^{p^s}=1} }\,(q-\zeta)  = [p^s]_q = \dfrac{1-q^{p^s}}{1-q} = \phi_p(q) \phi_p(q^p) \dots \phi_p(q^{p^{s-1}})
$$
where $\phi_p(q)=1+q+\dots + q^{p-1}$ is the $p$ - th cyclotomic polynomial. Theorem \ref{newthm} therefore gives:
\bean \label{qrat}
\Fr(h,\aa,q, z) = \Fr(h^{(l)},\aa^{(l)},q, z) + [p^l]_q \, (\dots)
\eean
where $\Fr(h^{(l)},\aa^{(l)},q, z) \in \Omega(z)$  and $\dots$ stands for a power series in $z$ with coefficients regular at $q$ given by $p^l$-th unit roots. Since by our assummption $q=1+t$ with $|t|_p<1$ we have   
$[p^l]_q \equiv 0 \pmod{p^s}$ for sufficiently large $l$, and therefore (\ref{qrat}) implies that
$
\Fr(h,\aa,q, z) \pmod{p^s}
$
is in $\Omega(z)$. As in Lemma \ref{lemmagaussnorm}, this implies that  $\Fr(h,\aa,q, z) \in \mathsf{GL_n}(\hat{E}_p)$. 
\end{proof}

\begin{rem}
In the scalar case case $n=1$ a similar result was recently obtained in \cite{VM22}. 
\end{rem}

\section{Limit to cohomology \label{cohomsec}} 
By Theorem \ref{newthm} from the previous section at  $q=\zeta$, where $\zeta\neq 1$ is a unit root satisfying $\zeta^{p^s}=1$, the Frobenius intertwiner $\Fr(h,\aa,q,z)$  specializes to a rational function of $z$. In this section we treat the remaining unit root $\zeta=1$. 

For this, let us consider the sequence 
$q=1+p^i,  \ \ i=1,2,\dots$
By Theorem \ref{mainthm}, for all $q$ in this sequence $U(h,\aa,q,z)$ is a Frobenius intertwiner.  Moreover, since all these $q$ as rational $U(h,\aa,q,z)\in \mathsf{GL}_n(E_p)$ where $E_p$ is the field of rational analytic elements, i.e., is the  completion of $\mathbb{Q}(z)$ with respect to the Gauss norm. The limit $i\to \infty$ corresponds to the limit $q\to 1$ in the $p$-adic norm. 
In turn, it is well-known that in the limit $q\to 1$ the fundamental solution matrix $\Psi(h,\aa,q,z)$ of the quantum $q$-difference equation converges to the fundamental solution of the quantum {\it differential} equation of $X$, e.g., note that the $q$-hypergeometric functions (\ref{hypereomfun}) specialize to the standard hypergeometric series in this limit.

Thus, in the limit $q\to 1$ the intertwiner $U(h,\aa,q,z)$ converges to the Frobenius intertwiner for the hypergeometric differential equation. In this way we arrive at the classical results of Dwork \cite{Dw69} on the Frobenius structures for the hypergeometric differential equations. In this Section we summarize the outcome of this limiting procedure.

\subsection{} 
As $q\to 1$, the quantum difference equation of $X$ degenerates to the quantum differential equation, which governs  the quantum cohomology of $X$. In this limit, the equivariant $K$-theory ring (\ref{ktheorC}) is replaced by the equivariant cohomology ring
\bean \label{cohom}
H^{\bullet}_{T}(X) = \matZ[x,a_1,\dots,a_n,h]/(x-a_1)\dots(x-a_n).
\eean
and the gamma class (\ref{gammaclass}) by:
\bean \label{gamcohom1}
\Gamma(h,\aa,x,z)=z^{x}\,  \prod\limits_{i=1}^{n} \dfrac{\Gamma(x+h-a_i)}{\Gamma(x+1-a_i) \Gamma(h)}  \in \hat{H}^{\bullet}_{T}(X).
\eean
where $\Gamma$ denotes the classical Gamma function, and $\hat{H}^{\bullet}_{T}(X)$ is the appropriate completion of the cohomology ring containing the gamma functions. The normalized  vertex function (\ref{normvert}) (also known as $J$-function in quantum cohomology) now equals:
\bean \label{cohomver}
\widetilde{V}(h,\aa,x,z) = \sum\limits_{d=0}^{\infty}\, \Gamma(h,\aa,x+d,z) \, \in \, \hat{H}^{\bullet}_{T}(X)[[z]].
\eean
A direct computation shows that this function satisfies the differential equation
\bean \label{hypergeq}
P(\aa,h,z)  \widetilde{V}(h,\aa,x,z) = (x-a_1)\dots(x-a_n) \Gamma(h,\aa,x,z) =0
\eean
(the last equality is because of the relation in (\ref{cohom})), where
\bean \label{hypergeomoperat}
P(\aa,h,z)  = z \prod\limits_{i=1}^{n}\,(D+h-a_i) - \prod\limits_{i=1}^{n}( D-a_i ), \ \ \  D = z \frac{d}{dz},
\eean
is the hypergeometric differential operator. The components of the vertex function (\ref{cohomver}) in any basis of $H^{\bullet}_{T}(X)$ provide a basis of solutions for the hypergeometric differential equation (\ref{hypergeq}). For example, the components of the vertex function (\ref{cohomver}) in the basis of $H^{*}_{T}(X)$ given by the $T$-fixed points $X^{T}$ are obtained by substituting $x=a_i$, $i=1,\dots,n$. In this way, we obtain a basis of solutions to (\ref{hypergeq}) given by the classical hypergeometric series:
$$
\widetilde{V}(h,\aa,a_i,z) = \Gamma(h,\aa,a_i,z) \, \sum\limits_{d=0}^{\infty} \, \left(\prod\limits_{j=1}^{n} \dfrac{(a_i+h-a_j)_d}{(a_i+1-a_j)_d}\right) z^d, \ \ i=1,\dots, n,
$$
where $(a)_d = a (a+1) \dots (a+d-1)$.

Let us denote $\Fr(h,\aa,z):=\left.\Fr(h,\aa,q,z)\right|_{q=1}$. 
By Theorem \ref{kobthm} $\Gamma_{p,q}(n) \stackrel{q\to 1}{\longrightarrow} \Gamma_{p}(n)$ and from Theorem \ref{mainthm} we obtain:
\begin{thm} \label{cohomainthem}
For any 
$\aa=(a_1,\dots,a_n) \in \matZ^{n}_p, \ \ h \in \matZ_p$ the power series $\Fr(h,\aa,z)$ takes values in $\mathsf{GL}_n(E_p)$ and gives a Frobenius intertwiner  from the hypergeometric system (\ref{hypergeomoperat}) with parameters $(a_1,\dots,a_n,h,z^p)$ to the hypergeometric system with parameters $(p a_1,\dots, p a_n ,p h,z)$. The constant term of this intertwiner is the operator acting in $\hat{H}^{\bullet}_{T}(X)$ as an operator of multiplication by the cohomology class class 
\bean \label{concohomterm}
\Fr(h,\aa, 0) = \prod\limits_{j=1}^{n} \, \dfrac{\Gamma_{p}(p x+ p h - p a_j )}{\Gamma_{p}(p x-p a_j ) \Gamma_{p}(p h )},
\eean
where $\Gamma_p$ is the Morita gamma function. 
\end{thm} 
The $q\to1$ limit   of Corollary \ref{fpcorrol} also gives:
\begin{cor}
If  $a_1,\dots, a_n$ are pairwise distinct, then the constant term of the Frobenius intertwiner $\Fr(h,\aa,z)$ is diagonal in the basis of $H^{*}_{T}(X)$ given by the classes of $T$-fixed points with the eigenvalues
\bean \label{frobconj}
\Fr(h,\aa,0)_{i,i}=\prod\limits_{j=1}^{n}\, \dfrac{\Gamma_{p}(p(a_i-a_j +h) )}{\Gamma_{p}(p(a_i-a_j)) \Gamma_{p}(p h )}, \ \ \ i=1,\dots, n.
\eean     
\end{cor}
In this way we arrive at the description of the Frobenius intertwines for the $p$-adic hypergeometric equations discovered by Dwork \cite{Dw69, Dw89} see also \cite{Ked21} where this result was recently revisited by other tools.

Recall that for $n\in p \matZ_p$ the Morita gamma function satisfies the identities $ 
\Gamma_p(n)=-\Gamma_p(1 + n) =1/\Gamma_p(-n)$. Using this, up to a constant scalar multiple independent of $x$,  (\ref{concohomterm}) can we written in the form
$$
\Fr(h,\aa, 0) = \prod\limits_{i=1}^{n} \, \Gamma_p(p x +p h -p a_j ) \Gamma_p(p a_j -  px )
$$
The arguments of the gamma functions in this product are exactly the $p$-multiples of the Chern roots of the tangent bundle $TX$, i.e., they are the Chern roots of  $TX^{(p)}:= \psi^p(TX)$ where $\psi^{p}$ is the $p$-th Adams operation.  We thus can write this formally as 
\begin{prop}
The constant term of the intertwiner $\Fr(h,\aa,\zz)$ acts in the equivariant cohomology $H^{*}_{T}(X)$ as multiplication by the cohomology class
\bean \label{genforms}
\Fr(h,\aa,0) = \Gamma_p(TX^{(p)})
\eean
\end{prop}
Here, the cohomology class $\Gamma_p({\mathcal{V}})$ of a vector bundle $\mathcal{V}$ is defined as in (\ref{KthGclass}) with $\Gamma_{p,q}$ replaced by the Morita gamma $\Gamma_p$. 
We expect that (\ref{genforms}) holds for the constant term of the Frobenius intertwiner of the quantum differential equations associated with more general varieties than $X$. The last Proposition is the cohomological limit of Proposition \ref{propKthCT}.

\section{Limit to $X=\mathbb{P}^{n-1}$: Frobenius structure for generalized Bessel equations \label{Pnsection}}
\subsection{} 
The equivariant quantum differential equation and its solutions for $X=\mathbb{P}^{n-1}$
arise as the limit $\hbar\to \infty$ of the corresponding equations and solutions for $T^{*} \mathbb{P}^{n-1}$ (we recall that the equivariant parameter $\hbar$ is the character of $T$ corresponding 
to the dilation of the cotangent directions of $T^{*}\mathbb{P}^{n-1}$ under the $T$-action) \cite{SV24}. Thus, the formal limit $\hbar\to \infty$ of our results in the previous section
gives the Frobenius intertwiners for the $T$-equivariant quantum differential equations of projective spaces.

Further, in the non-equivariant limit, when we ignore the $T$-action on $X=\mathbb{P}^{n-1}$, the quantum differential equation is nothing but the {\it generalized Bessel equation} (and the classical Bessel differential equation for $n=2$). The non-equivariant limit corresponds to the case $\aa=(0,\dots,0)$. Specializing at this point, we arrive at the Frobenius structure for Bessel equations studied in \cite{Dw74,Sp80}. This section is a summary of these limits.

\subsection{} 
As before, let $T=(\matC^{\times})^{n}$ be the torus acting on $\matC^{n}$ by scaling the coordinate lines with characters $a_1,\dots, a_n$. We consider the induced action on $X=\mathbb{P}^{n-1} =\matP(\matC^n)$. The equivariant cohomology ring has the form
\bean \label{projcohom}
H^{\bullet}_{T}(X) = \matZ[x,a_1,\dots,a_n]/(x-a_1)\dots (x-a_n).
\eean
Taking the formal limit $\hbar\to \infty$ of (\ref{gamcohom1}) we obtain the gamma class for $X=\mathbb{P}^{n-1}$:
$$
\Gamma(x,\aa,z) = \dfrac{z^{x}}{\prod\limits_{i=1}^{n} \Gamma(x+1-a_i)} \in \hat{H}^{\bullet}_{T}(X).
$$
The normalized equivariant cohomological vertex function (the $J$-function) of $\mathbb{P}^{n-1}$ equals:
$$
\widetilde{V}(x,\aa,z) = \sum\limits_{d=0}^{\infty}\, \Gamma(x+d,\aa,z) \, \in \, \hat{H}^{\bullet}_{T}(X)[[z]].
$$
More explicitly
$$
\widetilde{V}(x,\aa,z)= \dfrac{z^{x}}{\prod\limits_{i=1}^{n} \Gamma(x+1-a_i)} \, \sum\limits_{d=0}^{\infty}\, \dfrac{z^{d}}{(x-a_1+1)_d \dots (x-a_n+1)_d}.
$$
This power series satisfies the differential equation
$$
P(\aa,z) \widetilde{V}(x,\aa,z) =(x-a_1)\dots (x-a_n) \Gamma(x,\aa,z),
$$
where  $P(\aa,z)$ is the differential operator
\bean \label{equivbess}
P(\aa,z) =  z - (D-a_1)\cdots (D-a_n), \ \ \ D= z \dfrac{d}{dz}
\eean
and therefore in the ring (\ref{projcohom}) we have
\bean \label{eqbasspde}
P(\aa,z) \widetilde{V}(x,\aa,z) =0. 
\eean
Let $\Fr(h,\aa,z)$ be as in Theorem \ref{cohomainthem} and let 
$\Fr(\aa,z)$ denotes its limit as $\hbar\to\infty$. Then we obtain
\begin{thm} \label{Pnmainthm}
For any 
$\aa=(a_1,\dots,a_n) \in \matZ^{n}_p$, $\Fr(\aa,z)$ gives  a Frobenius intertwiner between solutions of differential equation (\ref{eqbasspde}) with parameters $(a_1,\dots,a_n,z^p)$ and parameters $(p a_1,\dots, p a_n ,z)$. The constant term of the intertwiner is represented by the operator acting in the completion of  (\ref{projcohom}) as an operator of multiplication by the cohomology class 
\bean \label{concohomtermPn}
\Fr(\aa,0) = \prod\limits_{j=1}^{n} {\Gamma_{p}}(p x-p a_j  ). 
\eean    
\end{thm}
We note that the arguments of the gamma functions in the product (\ref{concohomtermPn}) are the $p$ multiples of the Chern roots of the tangent bundle $TX$.  Thus, the universal formula (\ref{genforms}) still holds.
\subsection{}
At $\aa=(0,\dots,0)$ the equivariant cohomology of $X=\mathbb{P}^{n-1}$ specializes at the usual (non-equivariant) cohomology:
\bean \label{cohomusual}
H^{\bullet}(X) = \matZ[x]/(x)^n.
\eean
The gamma class and the vertex function have the form
$$
\Gamma(x,z) = \dfrac{z^{x}}{ \Gamma(x+1)^n},
$$
and 
\bean \label{verpnm}
\widetilde{V}(x,z) = \sum\limits_{d=0}^{\infty}\, \Gamma(x+d,z) =  \dfrac{z^{x}}{ \Gamma(x+1)^n} \, \sum\limits_{d=0}^{\infty}\, \dfrac{z^d}{(x+1)^n_d},
\eean
respectively. The vertex function solves
$$
P(z) \widetilde{V}(x,z) = x^n \,\Gamma(x,z) = 0
$$
where 
\bean \label{beseldiffop}
P(z)  = z - D^n
\eean 
The differential equation defined by (\ref{beseldiffop}) is the {\it generalized Bessel equation}. For $n=2$ this is the classical Bessel equation. The components of $\widetilde{V}(x,z)$ in any basis of cohomology (\ref{cohomusual}) provide a basis of solutions for these equations. 
In this case, the only natural basis is given by the powers of the hyperplane class $x$: 
\bean \label{cohomBas}
\textrm{basis of} \ H^{\bullet}(X)  = \{ 1,x,x^2,\dots, x^{n-1} \}.
\eean
In order to compute the corresponding components it is enough to expand (\ref{verpnm}) in the Taylor expansion up to degree $x^n$:
$$
\widetilde{V}(x,z)= \widetilde{V}_0(z) + \widetilde{V}_1(z) x +\dots +\widetilde{V}_{n-1} x^{n-1} 
$$
The functions $\widetilde{V}_0(z),\dots, \widetilde{V}_n(z)$ provide a basis of solutions to the generalized Bessel equation.
Note that (\ref{verpnm}) includes a multiple $z^{x}$ which has the following expansion:
$$
z^{x} = \sum\limits_{m=0}^{n-1}\, \dfrac{\ln(z)^m}{m!} x^m
$$
and thus the solutions may be written in the form

\bean \ \ \ \ \ \ \ \ \label{vertvs}
\begin{tiny}
 (\widetilde{V}_0(z),\dots, \widetilde{V}_{n-1}(z)) = ({V}_0(z),\dots, {V}_{n-1}(z)) \cdot 
\left(\begin{array}{ccccc}
1&\ln(z)&\ln(z)^2/2&..&\ln(z)^{n-1}/(n-1)! \\ 
0&1&\ln(z) &..&  \ln(z)^{n-2}/(n-2)!\\
0&0&1&..&\dots\\
..&..&..&..&\dots\\
0&0&0&..&1
\end{array}\right)
\end{tiny}
\eean
where $V_i(z)$ are some power series in $z$. In particular, $\tilde{V}_0(z)$ does not contain logarithmic terms. For $n=2$, $\tilde{V}_0(z)$ is the power series representing the classical Bessel function of the first kind. 
\subsection{} When all $a_i=0$, the  Theorem \ref{Pnmainthm} specializes to:
\begin{thm} \label{Pnmainthm}
The generalized Bessel equation admits Frobenius intertwiner $z\to z^p$ with the constant term acting in $H^{\bullet}(X)$ as the operator of multiplication by the cohomology class
\bean 
\Fr(0) =  {\Gamma_{p}}( p x )^{n} 
\eean    
\end{thm}

By construction, the intertwiner of Theorem \ref{Pnmainthm} maps the solutions $\tilde{V}(x,z^p)$ to the solutions $\tilde{V}(p x,z)$. Let $p^{\mathrm{deg}}$ be the operator of the cohomological grading, acting on the cohomology by $p^{\mathrm{deg}}(x^{k}) = p^{k} x^{k}$. The above theorem says that
$$
{\Gamma_{p}}(x )^{n}   \cdot p^{\mathrm{deg}}
$$
is the zero term of the intertwiner between $\tilde{V}(x,z^p)$ and $\tilde{V}(x,z)$ (note that as in (\ref{vertvs}), the operators act by multiplication from the right in our conventions). In other words this gives  an {\it automorphism} of the corresponding differential equation. We thus conclude with 
 \begin{thm} \label{Pnmainthm}
The generalized Bessel equation admits unique Frobenius automorphism acting by $z\to z^p$. The constant term of the automorphsim acts in cohomology as the operator
\bean  \label{automoper}
\Fr(0) =  {\Gamma_{p}}( x )^{n}  \cdot p^{\mathrm{deg}}.
\eean    
\end{thm}
Let us compute the matrix (\ref{automoper}) explicitly in the basis (\ref{cohomBas}). If
$\Psi(x,z)=\Psi_0(z)+\Psi_1(z) x+\dots + \Psi_{n-1} x^{n-1}$, then the convention is that the matrices of the operators act on the row vector
$
(\Psi_0(z),\Psi_1(z),\dots , \Psi_{n-1})
$
by multiplication from the right. For instance, multiplication by $x$ acts by $x\Psi(x,z)=\Psi_0(z) x+\Psi_1(z) x^2+\dots + \Psi_{n-2} x^{n-1}$, and thus the corresponding matrix is
$$
x= \left(\begin{array}{cccccc}
0&1 & 0&0&\dots & 0\\
0&0&1&0& \dots &  0 \\
0&0& 0 &  1& \dots &  0 \\
..&..&..&..& \dots & 1 \\
0&0&0&0& \dots & 0 \\
\end{array}\right).
$$
Let us recall that the logarithm of the Morita gamma function has the following Taylor expansion, see for instance Sections 58 and 61 in  \cite{Sh}:
$$
\log \Gamma_p(x) =\Gamma_p'(0) x - \sum\limits_{m\geq 2} \, \dfrac{\zeta_p(m)}{m}\, x^m
$$
where $\zeta_p(m)$ denote the values of $p$-adic zeta function. It is known that $\zeta_p(m)=0$ for even $m$, so only the odd powers of $x$ appear in the expansion. 
Thus, in  our conventions we obtain that the operator of multiplication by $\log \Gamma_p(x)^n$ is given by the $n\times n$  matrix
$$
\log \Gamma_p(x)^n  = \left(\begin{array}{cccccc}
0&\nu_1 & 0& \nu_3&\dots &  \dots\\
0&0& \nu_1&  0& \nu_3 &  \dots \\
0&0& 0 &  \nu_1& \dots &  \dots \\
..&..&..&..& \dots & 0\\
..&..&..&..& \dots & \nu_1 \\
0&0&0&0& \dots & 0 \\
\end{array}\right),
$$
where we denoted $\nu_1=n \Gamma_p'(0)$ and  $\nu_m = - n\, \zeta_p(m)/m$ for $m> 2$. We also obviously have
$$
p^{\mathrm{deg}}= \left(\begin{array}{cccccc}
1&0 & 0&0&\dots & 0\\
0&p&0&0& \dots &  0 \\
0&0& p^2 &0 & \dots &  0 \\
..&..&..&..& \dots & 0 \\
..&..&..&..& p^{n-2} & 0 \\
0&0&0&0& \dots & p^{n-1} \\
\end{array}\right).
$$
Finally, we obtain that the constant term of the automorphism has the following matrix:
$$
\Fr(0) =  {\Gamma_{p}}( x )^{n}  \cdot p^{\mathrm{deg}} =  \exp\left(\begin{array}{cccccc}
0&\nu_1 & 0& \nu_3&\dots &  \dots\\
0&0& \nu_1&  0& \nu_3 &  \dots \\
0&0& 0 &  \nu_1& \dots &  \dots \\
..&..&..&..& \dots & 0\\
..&..&..&..& \dots & \nu_1 \\
0&0&0&0& \dots & 0 \\
\end{array}\right) \cdot \left(\begin{array}{cccccc}
1&0 & 0&0&\dots & 0\\
0&p&0&0& \dots &  0 \\
0&0& p^2 &0 & \dots &  0 \\
..&..&..&..& \dots & 0 \\
..&..&..&..& p^{n-2} & 0 \\
0&0&0&0& \dots & p^{n-1} \\
\end{array}\right)
$$
and the exponent of a nilpotent operator is, of course, well defined. 
\begin{example}
For the case of the classical Bessel equation, corresponding to $n=2$, this matrix has the form
\bean \label{zertermexam}
\Fr(0) = \left(\begin{array}{ll}
1,& p\, \Gamma'_p(0)\\
0,& p 
\end{array}\right).
\eean
The existence of the Frobenius automorphism for the Bessel equation was discovered by Dwork  \cite{Dw74}. In particular, in Lemma 3.2 of \cite{Dw74} he gave an explicit formula for the matrix $\Fr(0)$ in terms of certain converging $p$-adic series. One checks that Dwork's matrix $\mathfrak{U}(0)$ coincides with (\ref{zertermexam}) after transposition. In particular, his series for the coefficient $\mathfrak{U}_{2,1}(0)=\gamma$ converges to $p \Gamma'_p(0)$. 

The action of the Frobenius automorphism for the generalized Bessel equations ($n>2$) was further studied by Sperber in \cite{Sp80}. The results of this Section are also in full agreement with these results.
\end{example}

We note that the values of $p$-adic zeta function appear naturally in our analysis of hypergeometric equations when we pass from the equivariant cohomology to the non-equivariant cohomology. Namely, the values of $p$-adic zeta functions appear as the coefficients of the expansion of the cohomology class $\Gamma_p(TX^{(p)})$ is some basis of $H^{*}(X)$. We expect the values of  $p$-adic zeta appearing in the study of the mirror symmetry for more general examples of $X$, e.g., \cite{BV23}, should be treated in the same way.

\section{Numerical examples of Theorem \ref{mainthm} \label{numsec}}

\subsection{} 
Let ${\Psi}(h,\aa,q,z)$ be the  fundamental solution matrix (\ref{fundsolsmat}). For simplicity we assume that all parameters are rational.  As in (\ref{frobstru}) we look for the Frobenius intertwiner matrix in the form
\bean \label{veryrel}
\Fr(z)={\Psi}(p h ,p \aa ,q,z) \Lambda {\Psi}(h ,\aa,q^p,z^p)^{-1}.
\eean
where $\Lambda$ is a constant matrix. 
In order to be a Frobenius intertwiner, we need to find  $\Lambda$ for which  (\ref{veryrel}) is an element of $\mathsf{GL}_n(E_p)$.  Lemma \ref{lemmagaussnorm} says that $\Fr(z)$ must be a power series
$$
\Fr(z)= \Fr_0+\Fr_1 z + \dots \in  \mathsf{Mat}_{n}(\matQ) [[z]] 
$$
such that for any natural number $s$, there exists a natural number  $m_s$ for which\footnote{Lemma \ref{lemmagaussnorm} says that $\Fr(z) \pmod{p^s}$ must be a rational function. Denominators of these rational functions are controlled by the singularities of the corresponding differential equations. In our case, the  qde has a singularity at $z=1$ and therefore the denominators must be powers of $(1-z)$. }
\bean \label{rationality}
(1-z)^{m_s} \, \Fr(z) \pmod{p^s} \in \mathsf{Mat}_{n}(\matZ)[z].
\eean
In words: after multiplication by sufficiently large power of $(1-z)$ the matrix elements of $\Fr(z)$ reduced $\pmod{p^s}$ truncate to polynomials in $z$. This condition is very strong and can be used to effectively discover non-trivial Frobenius intertwiners as we are about to show.

\subsection{} 
Note that if the parameters $\aa=(a_1,\dots, a_n)$ are pairwise distinct the matrix $\Lambda$ must be diagonal. Indeed, the fundamental solution matrix ${\Psi}(\aa,\hbar,q,z)$ contains logarithmic terms given by matrix $\mathsf{E}(z)$ in (\ref{analitfund}).   The intertwiner (\ref{veryrel}) is a power series in $z$ only if these log terms are canceled. It is elementary to check that for the pairwise distinct $a_i$ this happens only if $\Lambda$ is diagonal. Since $\Lambda$ is defined up to a scalar factor, we may assume further that $\Lambda_{n,n}=1$ so
\bean \label{consmat}
\Lambda=\left(\begin{array}{ccccc}
c_1 & 0& 0& .. & 0\\
0 & c_2 &0& .. & 0\\
..&..&..&..&..\\
0 & 0& ..& c_{n-1} & 0\\
0&0&0&..&1
\end{array}\right)
\eean 
for some unknown $c_1,\dots, c_{n-1} \in \matZ_p$. 
\subsection{} Let us assume $n=2$. In this case 
\bean \label{lamtwo}
\Lambda = \left(\begin{array}{cc}
c&0\\
0&1
\end{array}\right).
\eean 
Let us pick a prime $p=3$, set
$q=1+p = 4$,
and pick some random values of $a_1,a_2,h\in\matZ_p$, for instance
\bean \label{numval}
h=\dfrac{1}{2}, \ \ a_1=\dfrac{1}{5}, \ \  a_2=0.
\eean
(Since the solutions  only depend on the difference $a_1-a_2$ we can always assume $a_2=0$). For our computer check we use integral approximations of these numbers say up to $p^7$:
$$
\bar{h} = h \pmod{3^7} = 1094, \ \ \bar{a_1} = a_1 \pmod{3^7} = 875,  \ \ \ \bar{a_2}=0.
$$
We now determine the constant $c \in \matZ_p$ in (\ref{lamtwo}) from the rationality condition (\ref{rationality}) in a chain of steps.  For this we assume that $c$ has the following $p$-adic expansion:
$$
c= c^{(0)} + c^{(1)} 3 + c^{(2)} 3^2 + \dots
$$
where $c^{(i)} \in \{0,1,2\}$.

\vspace{3mm}
\underline{Step 1: determination of $c^{(0)}$ and $c^{(1)}.$}
\vspace{3mm}

\noindent Using a computer we approximate the power series ${\Psi}(\aa,\hbar,q,z)$ by series expansion (\ref{fundsolsmat}) up to $O(z^{20})$. 
Substituting this approximation to (\ref{veryrel}) and trying the integers $0,1,2$ for the values of $c^{(0)}$ and $c^{(1)}$  we find that
$
\Fr(z) \pmod{p} 
$
exists (i.e. no powers of $p$ in the denominators) only if $c^{(0)}=1$ and $c^{(1)}=0$, in which case we find that
\bean \label{modp1}
\Fr(z) \pmod{p}  =  \left( \begin {array}{cc} 1+O \left( {z}^{20} \right) &O \left( {z
}^{20} \right) \\ \noalign{\medskip}O \left( {z}^{20} \right) &1+O
 \left( {z}^{20} \right) \end {array} \right)
\eean 
We note that $\Fr(z) \pmod{p}$  given by (\ref{modp1}) does not depend on the value of $c^{(2)}$ and higher. 

\vspace{3mm}

\underline{Step 2: determination of $c^{(2)}.$}

\vspace{3mm}

\noindent
If $c^{(2)}=0$, using a computer we find
$$
\Fr(z) \pmod{p}  =  \left[ \begin {array}{cc} 1+O \left( {z}^{20} \right) &O \left( {z
}^{20} \right) \\ \noalign{\medskip}{z}^{12}+2\,{z}^{15}+O \left( {z
}^{20} \right) &1+O \left( {z}^{20} \right) \end {array} \right] 
$$
If $c^{(2)}=1$ we find
$$
\Fr(z) \pmod{p}  =  \left[ \begin {array}{cc} 1+O \left( {z}^{20} \right) &O \left( {z
}^{20} \right) \\ \noalign{\medskip}O \left( {z}^{20} \right) &1+O
 \left( {z}^{20} \right) \end {array} \right] 
$$
If $c^{(2)}=2$ we find
$$
\Fr(z) \pmod{p}  =  \left[ \begin {array}{cc} 1+O \left( {z}^{20} \right) &O \left( {z
}^{20} \right) \\ \noalign{\medskip}2\,{z}^{12}+{z}^{15}+O \left( {z
}^{20} \right) &1+O \left( {z}^{20} \right) \end {array} \right] 
$$
We see that the only value of $c^{(2)}$ for which the result is a polynomial in $z$ is $c^{(2)}=1$. These results do not depend on the value of the coefficient $c^{(3)}$ and higher  which are to be determined in the next step.

\vspace{3mm}

\underline{Step 3: determination of $c^{(3)}.$}
\vspace{3mm}

\noindent
If $c^{(3)}=0$ we find:
\begin{tiny}
$$
(1-z)^{2} \Fr(z) \!\! \pmod{p^2} = \left[ \begin {array}{cc} 1+z+{z}^{2}+O \left( {z}^{20} \right) &O
 \left( {z}^{20} \right) \\ \noalign{\medskip}6\,z+3\,{z}^{12}+3\,{z
}^{13}+3\,{z}^{14}+6\,{z}^{15}+6\,{z}^{16}+6\,{z}^{17}+O \left( {z}^{
20} \right) &1+7\,z+{z}^{2}+O \left( {z}^{20} \right) \end {array}
 \right] 
$$
\end{tiny}

\noindent
If $c^{(3)}=1$ we find:
\begin{tiny}
$$
(1-z)^{2} \Fr(z) \!\! \pmod{p^2} = \left[ \begin {array}{cc} 1+z+{z}^{2}+O \left( {z}^{20} \right) &O
 \left( {z}^{20} \right) \\ \noalign{\medskip}6\,z+O \left( {z}^{20}
 \right) &1+7\,z+{z}^{2}+O \left( {z}^{20} \right) \end {array}
 \right]
$$
\end{tiny}

\noindent
If $c^{(3)}=2$ we find:
\begin{tiny}
$$
(1-z)^{2} \Fr(z) \!\! \pmod{p^2} =  \left[ \begin {array}{cc} 1+z+{z}^{2}+O \left( {z}^{20} \right) &O
 \left( {z}^{20} \right) \\ \noalign{\medskip}6\,z+6\,{z}^{12}+6\,{z
}^{13}+6\,{z}^{14}+3\,{z}^{15}+3\,{z}^{16}+3\,{z}^{17}+O \left( {z}^{
20} \right) &1+7\,z+{z}^{2}+O \left( {z}^{20} \right) \end {array}
 \right]
$$
\end{tiny}
We conclude that there is unique value of $c^{(3)}=1$ for which the result is a polynomial in $z$. 

\vspace{3mm}

Continuing this process we will find that at $m$-th step there is always {\it unique} value of $c^{(m)}$ 
for which rationality condition (\ref{rationality}) is met. Up to the order $9$ these method gives the following value
\bean \label{experc}
c = 1+3^2+3^3+2\, 3^4+3^5+2 \,3^6+3^7+3^8 +O(3^{9}).
\eean
\subsection{}
Theorem \ref{mainthm} and Corollary \ref{fpcorrol} give the following formula for the constant $c$:
$$
c=\dfrac{\Lambda_{1,1}}{\Lambda_{2,2}}=\dfrac{\Gamma_{p,q}(p(a_1-a_2+h)) \Gamma_{p,q}(p(a_2-a_1))}{\Gamma_{p,q}(p(a_2-a_1+h)) \Gamma_{p,q}(p(a_2-a_1))}
$$
at (\ref{numval}) we have
$$
c=\dfrac{\Gamma_{p,q}(\frac{21}{10}) \Gamma_{p,q}(-\frac{3}{5})}{\Gamma_{p,q}(\frac{9}{10}) \Gamma_{p,q}(\frac{3}{5})}.
$$
For $p=3$ and $q=1+p=4$ we compute:
$$
\dfrac{\Gamma_{p,q}(\frac{21}{10}) \Gamma_{p,q}(-\frac{3}{5})}{\Gamma_{p,q}(\frac{9}{10}) \Gamma_{p,q}(\frac{3}{5})}= 1+3^2+3^3+2\, 3^4+3^5+2 \,3^6+3^7+3^8 +O(3^{9})
$$
in full agreement with the value obtained from the numerical experiment (\ref{experc}). 

The author have performed the same numerical experiments for numerous other rational values of  parameters $a_1,a_2$ and $h$ in $\matZ_p$. In all these cases the results are in full agreement with Theorem \ref{mainthm}.
Same method can be used for ranks $n>2$, in which case we have to determine the numerical values of $n-1$ constants in (\ref{consmat}). Again, we find full agreement with Theorem~\ref{mainthm}.

\section{Concluding remarks and future directions \label{conclsec}}

\subsection{}
We expect that the Frobenius intertwiner can be defined enumeratively as a partition function counting the quasimaps from $\mathbb{P}^1$ to a Nakajima variety $X$.
For this, one can consider the quasimaps
with two different relative conditions at $0\in \mathbb{P}^1$ and $\infty \in \mathbb{P}^1$. The corresponding moduli space of quasimaps is represented by  Fig.\ref{quaspic}, we refer to \cite{Oko17} and also \cite{PSZ} where the pictorial notations for various quasimap moduli spaces were introduced. 

\begin{figure}[H]
\includegraphics[scale=0.3]{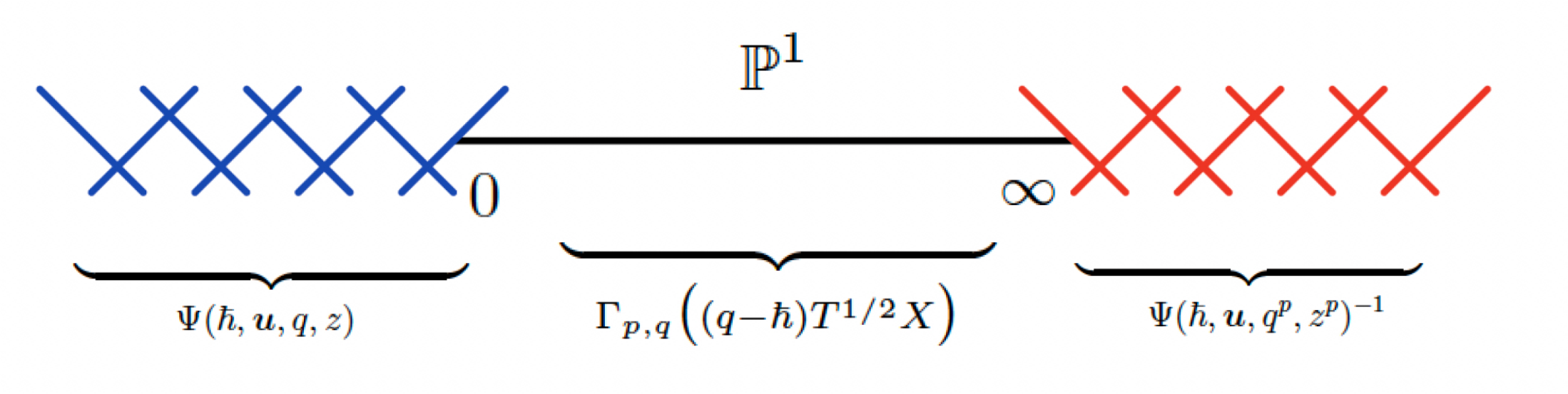}
\caption{Frobenius intertwiner as a partition function of relative quasimaps} \label{quaspic}
\end{figure}

At $0\in \mathbb{P}^1$ we have the usual relative condition on quasimaps as defined in \cite{Oko17}. We recall that the relative boundary condition at $0\in \mathbb{P}^1$ means that $\mathbb{P}^1$ is allowed to deform into a chain of rational ``bubbles'' with simple nodal intersections.  This chain of bubbles is represented by the blue color in Fig.\,\ref{quaspic}.

It is known, Theorem 8.1.16 in \cite{Oko17}, that the partition function of relative quasimaps from the chain of bubbles at $0\in \matP^1$ is a fundamental solution matrix $\Psi(\hbar,\uu,q,z)$ of the quantum difference equation associated with $X$, normalized so that
$$
\Psi(\hbar,\uu,q,z) = 1+ O(z) \in K_{T}(X)^{\otimes 2}[[z]],
$$
in terminology of \cite{Oko17} $\Psi(\uu,\hbar,q,z)$ is also referred to as the {\it capping operator}.

At $\infty \in \mathbb{P}^{1}$ we also impose a relative condition but require, in addition, the relative quasimaps to be $\matZ/p\matZ$-invariant. More precisely, recall that the moduli space of relative quasimaps is defined as a  stack quotient of quasimaps from a chain of $\mathbb{P}^1$'s by the automorphism group $(\matC^{\times})^{l}$ where $l$ is the length of the chain, see Section 6.2 in \cite{Oko17}. The moduli space of relative quasimaps at $\infty \in \mathbb{P}^{1}$ is constructed as follows: first we consider the space of quasimaps from a chain of rational curves of length $l$ which are invariant with respect to $\matZ/p\matZ$, which acts in each ``bubble'' (the component of the chain) via multiplication by $p$ - th roots of unity. Then we define the $\matZ/p\matZ$-invariant quasimaps as the stack quotient of this moduli space by $(\matC^{\times})^{l}$. This ``$\matZ/p\matZ$-invariant chain of bubbles'' is represented by the red color in Fig.\,\ref{quaspic}.

By the $\matZ/p\matZ$-invariance, the degrees of the quasimaps allowed in this picture are all multiples of $p$. In particular, the corresponding partition function is a power series in $z^p$. 
For appropriately defined moduli space, the corresponding partition function arising at $\infty \in \mathbb{P}^1$ is equal to
$\Psi(\uu,\hbar,q^p,z^p)^{-1}$.  The inverse arises due to transposition and inversion of character $q\to 1/q$ at $\infty \in \mathbb{P}^1$.

Finally, the contribution of degree $0$ - quasimaps which correspond to the ``constant'' quasimaps from the parameterized component $\matP^1$ can be computed by $\matC^{\times}_q$ - equivariant localization similarly to how it was done for the shift operators in Section 8.2 of \cite{Oko17}, in particular see Lemma 8.2.12. The resulting contribution has the form
$$
\dfrac{\Phi(p h,p\aa,p x,q,z)}{\Phi(h,\aa,x,q^p,z^p)} \sim \Gamma_{p,q}\Big( (q^p-\hbar^p) T^{1/2}X^{(p)} \Big),
$$
where $T^{1/2}X$ denotes the corresponding choice of the polarization bundle of $X$. The numerator of this fraction is the localization contribution at $0 \in \mathbb{P}^1$ and the denominator is at $\infty \in \mathbb{P}^1$. Combining all this together we obtain that the corresponding partition function is
$$
\Psi(\uu,\hbar,q,z)  \Gamma_{p,q}\Big( (q^p-\hbar^p) T^{1/2}X^{(p)} \Big) \Psi(\uu,\hbar,q^p,z^p)^{-1} 
$$
which is exactly the Frobenius intertwiner (\ref{froint}).

\subsection{} 
The idea of $\matZ/p \matZ$ - invariant  quasimap moduli is similar to the one used in \cite{HL24} to investigate the quantum Steenrod operators in quantum cohomology. It was shown in \cite{HL24}  that for a large class of varieties $X$, which includes the Nakajima varieties, the quantum Steenrod operations coincide with the $p$-curvature of the quantum connection.  

The $q$-difference connections also have well defined $p$-curvature. Namely, the analog of $p$-curvature of $q$-difference connection (\ref{qde}) is:
$$
\mathbf{C}^{(p)}(\zeta,z) =\left.\Big(\Mop(z q^{p-1}) \dots  \Mop(z q)\Mop(z)\Big)\right|_{q=\zeta}
$$
where $\zeta$ is a non-trivial unit root of order $p$. In \cite{KS} we showed that the $p$-curvature of the quantum connection of $X$ arises as the first non-trivial term in the $p$-adic expansion of $\mathbf{C}^{(p)}(\zeta,z)$.

Let us show that $p$-curvature $\mathbf{C}^{(p)}(\zeta,z)$ is similar to $\Mop(z^p)$. Indeed, iterating the $q$-difference equation (\ref{qde})  $p$-times we obtain:
$$
\Psi(q,z q^p)= \Mop(zq^{p-1}) \dots \Mop(z) \Psi(z)
$$
also, rising all variables in (\ref{qde}) to their $p$-th powers we obtain:
$$
\Psi(q^p,z^p q^p) =\Mop(z^p) \Psi(q^p,z^p)
$$
Dividing the first equation by the second, and specializing at $q=\zeta$ we arrive at:
$$
\Fr(\zeta,z) \Mop(z^p) \Fr(\zeta,z)^{-1} = \mathbf{C}^{(p)}(\zeta,z).  
$$
In particular, the spectrum of the $q$-deformed $p$-curvature $\mathbf{C}^{(p)}(\zeta,z)$ coincides with the spectrum of $\psi^p$-twisted quantum multiplication $\Mop(z^p)$. This result is a $q$-version of the isospectrality recently proven in \cite{EV24} 
in the differential case using different methods.

\section*{Acknowledgements}
I thank S.\,Bai, P.\,Etingof, V.\,Golyshev, K.\,Kedlaya, J. Hee Lee, A.\,Oblomkov, A.\,Petrov, D.\,Vargas-Montoya, M.~Vlasenko  and especially A.\,Okounkov for very inspiring discussions.
I thank A.\,Varchenko for introducing me to the subject. This note is the extended version of the talks \cite{Sm23} and \cite{Sm24}, I thank the organizers for these opportunities.  
This work is supported by NSF grants DMS - 2054527 and DMS-2401380.

\appendix

\section{$p$-adic $q$-deformations}
In this section we recall standard notations of $q$-calculus. We also obtain results about $p$-adic $q$-gamma function which we we need in this paper.  
\subsection{$q$-numbers} 
For  $n\in \matN$ we define the $q$-number by:
$$
[n]_q = \dfrac{1-q^n}{1-q} = 1+q+\dots + q^{n-1}.
$$
Let $[n]_q!=[1]_q [2]_q \dots [n]_q$ denote the $q$-factorial and 
\bean \label{qbinom}
\binom{n}{k}_{\!\!q}  = \dfrac{[n]_q!}{[k]_q! [n-k]_q!}
\eean
Let us define the reciprocal of the $q$-gamma functions as formal product:
$$
(x,q)_{\infty} = \prod\limits_{i=0}^{\infty} (1-x q^i)
$$
We also define finite products
$$
(x,q)_{d} = \dfrac{(x,q)_{\infty}}{(x q^d,q)_{\infty}} = (1-x) (1-x q) \dots (1-x q^{d-1})
$$
and 
\bean \label{sqbin}
[x,q]_{d} = (1-q x) \dots (1-q^{d-1} x) = \dfrac{(x q,q)_{\infty}}{(x q^{d},q)_{\infty}}
\eean

\subsection{$q$-hypergeometric series at unit roots} 
Let us recall the $q$-binomial theorem: 
$$
\dfrac{(z \hbar,q)_{\infty}}{(z,q)_{\infty}} = \sum\limits_{d=0}^{\infty} \,\dfrac{(\hbar,q)_d}{(q,q)_d} z^d
$$
At $\hbar=q^{l}$, $l\in \mathbb{N}$ it is convenient to write the coefficients as
\bean \label{poham1}
\dfrac{(\hbar,q)_d}{(q,q)_d} = \dfrac{\frac{(\hbar,q)_{\infty}}{(\hbar q^d,q)_{\infty}}}{\frac{(q ,q)_{\infty}}{(q  q^d,q)_{\infty}}}  = \dfrac{(\hbar,q)_{\infty}}{(q ,q)_{\infty}} \dfrac{(q  q^d,q)_{\infty}}{(\hbar q^d,q)_{\infty}}
\eean 
Using notation (\ref{sqbin}) at $\hbar =q^{l}$ we thus write:
\bean \label{poham2}
\dfrac{(q^{l},q)_d}{(q,q)_d} = \dfrac{[q^{d},q]_{l}}{[1,q]_{l}} 
\eean 
At the same time at $\hbar=q^{l}$ we have
$$
\dfrac{(z \hbar,q)_{\infty}}{(z,q)_{\infty}} = \dfrac{1}{(z,q)_{l}}
$$
Combining all this together, we arrive at the following version of the $q$-binomial theorem:
\bean \label{psexpan}
\dfrac{1}{(z,q)_{l}} = \sum\limits_{d=0}^{\infty}\, \dfrac{[q^{d},q]_{l}}{[1,q]_{l}} z^d 
\eean 
Note that the coefficients of the power series (\ref{psexpan}) are polynomials in $q$ and in particular, can be specialized at unit roots. Let $p$ be a prime and let $\zeta\neq 1$ be a  unit root $\zeta^{p^s}=1$ for some natural number $s$. We may assume that the order of $\zeta$ is $p^m$ with $1\leq m\leq s$. 
Let $h'$ be a natural number. Then, applying (\ref{psexpan}) for $l=h' p^s$ we have
$$
\left.\dfrac{1}{(z,q)_{h' p^s}}\right|_{q=\zeta}=\dfrac{1}{(1-z^{p^m})^{h' p^{s-m}}}  = \sum\limits_{i=0}^{\infty}\, \left.\dfrac{[q^i,q]_{h' p^s}}{[1,q]_{h'p^s}}\right|_{q=\zeta}  z^i 
$$
At the same time, since $\zeta^p$ is a primitive unit roots of order $p^{m-1}$ applying (\ref{psexpan}) for $l=h' p^{s-1}$ we find
$$
\left.\dfrac{1}{(z^p,q^p)_{h' p^{s-1}}}\right|_{q=\zeta}=\dfrac{1}{(1-z^{p^m})^{h' p^{s-m}}}   = \sum\limits_{i=0}^{\infty}\, \left.\dfrac{[q^{pi},q^{p}]_{h' p^{s-1}}}{[1,q^p]_{h'p^{s-1}}}\right|_{q=\zeta}  z^i 
$$
Comparing the coefficients we conclude
\begin{lem} \label{coeffcongr}
Let $\zeta\neq 1$ be a unit $\zeta^{p^s}=1$ and let $h'\in \mathbb{N}$ then for any $i\in \mathbb{N}$ we have
\bean \label{coffs}
\left.\dfrac{[q^{i},q]_{h' p^s}}{[1,q]_{h' p^s}}\right|_{q=\zeta}=\left.\dfrac{[q^{p i},q^p]_{h' p^{s-1}}}{[1,q^p]_{h' p^{s-1}}}\right|_{q=\zeta} 
\eean 
Coefficients (\ref{coffs}) are equal to $0$ unless $p\mid i$.
\end{lem}

\begin{lem} \label{hypergeomlem}

a) If $h \in \matZ$ then the hypergeometric function $F(h,\aa,q,z)$  (\ref{hypereomfun}) does not have poles at $q$ given by roots of unity.

b) Let $k \in \mathbb{Z}$ and $\zeta$ be a primitive unit root of order $m$ then:
$$
\left.F(m k,\aa,q,z)\right|_{q=\zeta} = \dfrac{1}{(1-z^m)^k} 
$$    
\end{lem}
\begin{proof}
The proof follows directly from the definition of $q$-hypergeometric series.     
\end{proof}

\subsection{$p$-adic $q$-deformations \label{pqgamm}} 
 We recall the original definitions of the $p$-adic gamma and $q$-gamma functions:
\begin{thm}[Koblitz, \cite{Ko80}] \label{kobthm}
For $q=1+t$ with $|t|_p<1$ the function defined on $n\in\mathbb{N}$ by
$
\Gamma_{p,q}(n)=(-1)^{n} \prod\limits_{{0<i<n}\atop {p \nmid  i}}\, [i]_q
$
extends to a continuous function on $\matZ_p$.

In the limit $t\to0$ we have 
$\lim\limits_{q\to 1}\, \Gamma_{p,q} = \Gamma_p$ where $\Gamma_p$ denotes the $p$-adic  gamma function of Morita \cite{Mo75}, which is the extension to $\matZ_p$ of 
$
\Gamma_{p}(n)=(-1)^{n} \prod\limits_{{0<i<n}\atop {p \nmid  i}}\, n.
$
\end{thm}

\begin{lem} \label{extlem}
For $q=1+t$ with $|t|_p<1$, the function of $n,h\in \matN$ given by 
\bean \label{qfactors}
f(h,n)=\dfrac{\dfrac{(q^{p h},q)_{\infty}}{(q,q)_{\infty}} 
\dfrac{(q^{pn +1},q)_{\infty}}{(q^{pn + ph},q)_{\infty}}}{\dfrac{(q^{p h},q^p)_{\infty}}{(q^p,q^p)_{\infty}} 
\dfrac{(q^{pn +p},q^p)_{\infty}}{(q^{pn + ph},q^p)_{\infty}}}
\eean
extends to $\matZ_p$ values of the arguments by $f(h,n)= \dfrac{\Gamma_{p,q}(pn +ph)}{\Gamma_{p,q}(pn)\, \Gamma_{p,q}(ph)}$.  
\end{lem}
\begin{proof}
If $n$ and $h$ are two natural numbers then
$$
\dfrac{(q^{p h},q)_{\infty}}{(q,q)_{\infty}} 
\dfrac{(q^{pn +1},q)_{\infty}}{(q^{pn + ph},q)_{\infty}} = \binom{p n+ p h -1}{pn}_{\!\!q} 
$$
is a $q$-binomial coefficient (\ref{qbinom}). Similarly
$$
\dfrac{(q^{p h},q^p)_{\infty}}{(q^p,q^p)_{\infty}} 
\dfrac{(q^{pn +p},q^p)_{\infty}}{(q^{pn + ph},q^p)_{\infty}} = \binom{n+ h -1}{n}_{\!\!q^p}.
$$
For the ratio of these factors we find:
$$
\dfrac{\binom{p n+ p h -1}{pn}_{\!\!q}}{\binom{n+ h -1}{n}_{\!\!q^p}} = \dfrac{\prod\limits_{{0<i<pn + ph} \atop {p\nmid i}}\, [i]_q}{(\prod\limits_{{0<i<pn} \atop {p\nmid i}}\, [i]_q) (\prod\limits_{{0<i<ph} \atop {p\nmid i}}\, [i]_q) } = \dfrac{\Gamma_{p,q}(pn +ph)}{\Gamma_{p,q}(pn)\, \Gamma_{p,q}(ph)}
$$
By Theorem \ref{kobthm} this function extends to $\mathbb{Z}_p$-values of $h$ and $n$.        
\end{proof}

Let us recall that the $p$-adic unit roots $q \in \Omega$ of order $p^s$, are of the form $q=1+t$ with 
$|t|_p = p^{-\frac{1}{p^{s-1}(p-1)}}<1$, \cite{Ko84}.  By Theorem \ref{kobthm} we can consider the values of $p$-adic $q$-gamma functions at such $q$. The following Lemma will be useful in the next section.

\begin{lem} \label{uonecorr} let $x,y \in \mathbb{Z}_p$ and $q \neq 1$ be a  $p^s$-th unit root, then 
\bean  \label{corrid}
\Gamma_{p,q}(x+y p^s) = \Gamma_{p,q}(x) \Gamma_{p,q}(y p^s)
\eean 
\end{lem}
\begin{proof}
First we note that for
$x \in \mathbb{Z}_p$ and $q \neq 1$ given by a  $p^s$-th unit root we have:
\bean \label{gamfact}
\Gamma_{p,q}(x+p^s) = \Gamma_{p,q}(x) \Gamma_{p,q}(p^s)
\eean 
Indeed, for $x \in \mathbb{N}$  the identify  (\ref{gamfact}) follows directly from the definition of $(p,q)$-Gamma function: 
$$
\Gamma_{p,q}(x)=(-1)^{n} \prod\limits_{{0<i<x}\atop {p \nmid  i}}\, [i]_q
$$
which extends to $x \in \mathbb{Z}_p$ as in  Theorem \ref{kobthm}.   Next, for $y \in \mathbb{N}$ (\ref{corrid}) follows by applying (\ref{gamfact}) $y$-times. Then both sides of (\ref{corrid}) extend to  $y \in \mathbb{Z}_p$ as in Theorem \ref{kobthm} again.  
\end{proof}

\subsection{}
We can write
$$
\Gamma_{p,q}(p) = - \dfrac{1-q}{1-q}\dfrac{1-q^2}{1-q}\cdots \dfrac{1-q^{p-1}}{1-q} = - \lim\limits_{w\rightarrow 1}\, \dfrac{1-w}{1-w} \dfrac{1-w q}{1-q} \cdots \dfrac{1-w q^{p-1}}{1-q} 
$$
where the limit is for regularization purposes. 
Now let us assume that $q$ is a primitive $p$-th root of unity, $q^p=1$ then from above formula we obtain:
$$
\Gamma_{p,q}(p) = \dfrac{-1}{(1-q)^{p-1}} \lim\limits_{w\rightarrow 1}\, \dfrac{1-w^p}{1-w} = \dfrac{-p}{(1-q)^{p-1}}
$$
Using same logic we also find for $m\in \matN$:
$$
\Gamma_{p,q}(m p) = \Gamma_{p,q}(p)^m = \Big( \dfrac{-p}{(1-q)^{p-1}} \Big)^m
$$
For primitive $p^s$-th we can use a similar trick. First we write:
$$
\Gamma_{p,q}(p^s)  = - \prod\limits_{{0<i<p^s}\atop {p\nmid i}} \, \dfrac{1-q^i}{1-q} = \dfrac{-1}{(1-q)^{p^{s-1}(p-1)}} \lim\limits_{w\to 1} \, \dfrac{\prod\limits_{{0<i<p^s}}(1-w q^i)}{\prod\limits_{{0<i<p^{s-1}}}(1-w q^{i p})}
$$
Assume $q$ is a primitive $p^s$-th root of unity, $q^{p^s}=1$, then the previous expression gives:
$$
\Gamma_{p,q}(p^s)  =  \dfrac{-1}{(1-q)^{p^{s-1}(p-1)}} \lim\limits_{w\to 1} \dfrac{1-w^{p^s}}{1-w^{p^{s-1}}} = \dfrac{-p}{(1-q)^{p^{s-1}(p-1)}}
$$
In general we conclude
\begin{prop}
Let $q$ be a primitive $p^s$-th unit root and $m\in \matN$ then
$$
\Gamma_{p,q}(m p^s)  = \Big( \dfrac{-p}{(1-q)^{p^{s-1}(p-1)}} \Big)^m  
$$
\end{prop}
Note also that a primitive $p^s$-th unit root has the form
$
q=1+t
$
with $|t|_p = p^{-\frac{1}{p^{s-1}(p-1)}}$. Thus we see that $|\Gamma_{p,q}(p^s)|_p=1$.

\end{document}